\RequirePackage{fix-cm}
\documentclass[smallextended]{svjour3mod}       
\smartqed  
\usepackage{graphicx}
\usepackage{amstext}
\usepackage{bm}
\usepackage{verbatim}
\usepackage{amsmath}
\usepackage{amsfonts}
\usepackage{units}
\usepackage{mathrsfs}
\usepackage{cite}
\usepackage{amssymb}

\newcommand{\p}{\partial}

\newcommand{\dd}{{\rm d}}

\begin{document}

\title{Convex neighborhoods for Lipschitz connections and sprays\thanks{This work has been
partially supported by GNFM of INDAM.}
}


\author{E. Minguzzi }


\institute{E. Minguzzi \at
              Dipartimento di Matematica e Informatica ``U. Dini'', Universit\`a
degli Studi di Firenze, Via S. Marta 3,  I-50139 Firenze, Italy. \\
              \email{ettore.minguzzi@unifi.it}           
}

\date{}

\maketitle

\begin{abstract}
We establish that over a $C^{2,1}$ manifold the
exponential map of any Lipschitz  connection or spray determines a
local Lipeomophism and that, furthermore, reversible convex normal
neighborhoods do exist.
To that end we use the  method of Picard-Lindel\"of approximation to
prove the strong differentiability  of the exponential map at the
origin and hence a version of Gauss' Lemma which does not require
the differentiability of the exponential map. Contrary to naive differential degree counting,
the distance functions are shown to gain one degree and hence to be $C^{1,1}$.

As an application to
mathematical relativity, it is argued that the mentioned
differentiability conditions can be considered the optimal ones to
preserve most results of causality theory. This theory is
also shown to be generalizable to the Finsler spacetime case. In
particular, we prove that the local Lorentzian(-Finsler) length
maximization property of causal geodesics in the class of absolutely
continuous causal curves holds already for $C^{1,1}$ spacetime
metrics. Finally, we study the local existence of convex functions
and show that arbitrarily small globally hyperbolic convex normal
neighborhoods do exist.
\end{abstract}

\newpage


\begingroup \small
\contentsline {section}{\numberline {1}Introduction and results}{2}
\contentsline {subsection}{\numberline {1.1}Example: Pseudo-Finsler geometry}{3}
\contentsline {subsection}{\numberline {1.2}The exponential map for sprays}{5}
\contentsline {subsection}{\numberline {1.3}Gauss' Lemma for pseudo-Finsler sprays}{13}
\contentsline {subsection}{\numberline {1.4}Some applications to mathematical relativity}{17}
\contentsline {subsection}{\numberline {1.5}Distance balls are convex}{20}
\contentsline {subsection}{\numberline {1.6}Convexity on Lorentzian manifolds}{23}
\contentsline {subsection}{\numberline {1.7}Two variations on the main theme}{25}
\contentsline {subsection}{\numberline {1.8}Some technical preliminary results}{26}
\contentsline {section}{\numberline {2}Proofs I: A Picard-Lindel\"of analysis}{28}
\contentsline {subsection}{\numberline {2.1}Existence of geodesics}{31}
\contentsline {subsection}{\numberline {2.2}Lipschitz dependence on initial conditions {\relax (Theorem 1)}}{32}
\contentsline {subsection}{\numberline {2.3}Strong differentiability of exp {\relax  (Theorems 3 and 13)}}{34}
\contentsline {subsubsection}{\numberline {2.3.1}Normal vector bundle case {\relax (Theorem 13)}}{36}
\contentsline {subsection}{\numberline {2.4}Convex neighborhoods {\relax  (Theorem 4)}}{37}
\contentsline {subsection}{\numberline {2.5}Role of the coordinate affine structure and position vector}{38}
\contentsline {subsection}{\numberline {2.6}Local Lipeomorphisms {\relax (Theorems 3 and 13)}}{41}
\contentsline {section}{\numberline {3}Proofs II: Pseudo-Finsler sprays and connections}{42}
\contentsline {subsection}{\numberline {3.1}Gauss' Lemma {\relax  (Theorem 5)}}{42}
\contentsline {subsection}{\numberline {3.2}Local properties of geodesics in pseudo-Finsler geometry {\relax  (Theorem 6)}}{44}
\contentsline {subsection}{\numberline {3.3}Strong convexity of squared Riemannian distance {\relax  (Theorem 8)}}{47}
\contentsline {subsection}{\numberline {3.4}Splitting of the metric at a given radius {\relax (Theorem 9)}}{48}
\contentsline {subsection}{\numberline {3.5}Some local results on the strong concavity of the squared Lorentzian distance {\relax (Theorems 10, 11, Corollary 2)}}{49}
\endgroup 

\section{Introduction and results}

Let $M$ be an $n$-dimensional paracompact connected $C^{2,1}$
manifold and let $x^\mu\colon U\to \mathbb{R}^n$ be a local chart
where $U$ is an open subset. Every chart induces a chart $(x^\mu,
v^\mu)\colon \pi^{-1}(U)\to \mathbb{R}^n\times \mathbb{R}^n$, $\mu=0,1,\cdots,n-1$, on
the tangent bundle $\pi: TU\to U$.

For the moment let us consider a second order ODE defined just over
$U$ by
\begin{align}
\frac{\dd x^{\mu}}{ \dd t}&=v^\mu,  \label{fhs} \\
\frac{\dd v^\mu}{\dd t}&=H^\mu(x,v) \label{fha} ,
\end{align}
where $H^\mu$ is locally Lipschitz. Under a coordinate change
$\tilde{x}^\mu=\tilde{x}^\mu(x^\alpha)$ the system becomes
$\dot{\tilde{x}}^\mu=\tilde{v}^\mu$, $\dot{\tilde{v}}^\mu=
\tilde{H}^\mu(\tilde{x},\tilde{v})$, where
\begin{equation} \label{eic}
\tilde{H}^\mu(\tilde{x},\tilde{v})=\frac{\p \tilde{x}^\mu}{\p
x^\alpha \p x^\beta} \frac{\p x^\alpha}{\p \tilde{x}^\gamma}
\frac{\p x^\beta}{\p \tilde{x}^\delta}\, {\tilde{v}}^\gamma
{\tilde{v}}^\delta+ \frac{\p \tilde{x}^\mu}{\p x^\nu}\,
H^\nu(x(\tilde{x}), v(\tilde{x},{\tilde{v}})),
\end{equation}
and  $v^\mu(\tilde{x},{\tilde{v}})=\frac{\p x^\mu}{\p \tilde{x}^\nu}
\,\tilde{v}^\nu$.
The transformation shows that the following notions are well-defined as independent of the coordinate chart:
\begin{itemize}
\item[(a)] $H^\mu$ is positive homogeneous of second degree on the
velocities, that is, for every $s>0$, $H^\mu(x,sv)=s^2H^\mu(x,v)$,
\item[(b)] $H^\mu$ is a homogeneous quadratic form in the velocities.
\end{itemize}
In the former case we say that the second
order ODE defines a {\em spray} over $M$ \cite{ambrose60,lang99}
while in the latter more restrictive case we say that it defines a (torsionless) {\em
connection}. In the latter case we set
$H^\mu(x,v)=-\Gamma^\mu_{\alpha \beta}(x) v^\alpha v^\beta$ and we
recognize in Eq.\ (\ref{eic})  the transformation rule for the
Christoffel symbols $\Gamma^\mu_{\alpha \beta}$. The transformation
rule (\ref{eic}) clarifies that $M$ must be $C^{2,1}$ to make sense
of the Lipschitz condition on sprays. The non-stationary solutions
to Eqs.\ (\ref{fhs})-(\ref{fha}) will be called {\em geodesics}.

\begin{remark}
Actually, the above notion of spray is somewhat more general than
that introduced in \cite{ambrose60,lang99} since these authors drop
the condition $s>0$ on (a). Our definition, consistent with current
usage \cite{antonelli93}, allows us to include non-reversible
pseudo-Finsler manifolds in our analysis of sprays. Clearly, if
$H^\mu(x,v)$ defines a spray then $\tilde{H}^\mu(x,v):=H^\mu(x,-v)$
 also defines a  spray called the {\em reverse spray}. If
$H^\mu(x,v)=H^\mu(x,-v)$ the spray is called {\em reversible}. If
$x(t)$ is a geodesic then $x(-t)$ is a geodesic for the reverse
spray but not necessarily for the original spray. Thus the direction
of the parametrization of geodesics is  important. If we say
that two points are connected by a unique geodesic we tacitly assume
that the spray under consideration is $H^\mu(x,v)$; that claim does
not exclude the possible presence of a connecting geodesic, with
different image, for the reverse spray.
\end{remark}

From now on we shall consider just locally Lipschitz sprays and we
shall clearly speak of connection whenever $H^\mu$ is a quadratic
form. If $H^\mu$ is twice continuously differentiable with respect
to the velocities on the zero section of $TM$, then differentiating
twice $H^\mu(x,s v)=s^2H^\mu(x,v)$ with respect to $s$ and letting
$s\to 0$ we obtain $H^\mu(x,v)=\frac{1}{2}(\p^2 H^\mu/\p
v^\alpha \p v^\beta(x,0)) v^\alpha v^\beta$, that is the spray is a
connection \cite{lang99}. Whenever the connection comes from a
pseudo-Riemannian metric $g$ we shall assume $g$ to be $C_{loc}^{1,1}$.
For notational convenience, we shall write just $C^{k,1}$ for
$C^{k,1}_{loc}$. It is customary \cite{berestovskij93} to call {\em non-regular} the geometrical theory for which the differentiability condition on $g$ is weaker than $C^2$ (or that on the connection or spray is weaker than $C^1$). As an example of manifold which has a $C^{1,1}$ metric $g$, consider a cylinder closed by two semispherical cups, where the submanifold is endowed with the  metric induced from the Euclidean space.

\subsection{Example: Pseudo-Finsler geometry} \label{hbh}
Pseudo-Finsler geometry \cite{beem70,beem76b,akbarzadeh88} is a
generalization of Finsler geometry \cite{matsumoto86,bao00} in which
the fundamental tensor $g$ is required to be non-degenerate rather
than positive definite. Geodesics in pseudo-Finsler geometry are
described by sprays.

Since the definitions of pseudo-Finsler manifold which can be found
in the literature impose too strong differentiability conditions we
 provide a different definition.

\begin{definition}
A {\em pseudo-Finsler manifold} $(M,g)$ is a paracompact connected
$C^{2,1}$ manifold endowed with a $C^{1,1}$ symmetric  tensor \[g:
TM\backslash 0 \to T^*M\otimes_{M} T^*M, \quad (x,v) \mapsto g_{(x,v)},\]
defined on the non vanishing vectors, which is non-singular and
satisfies (in one and hence every chart induced from a chart on $M$)
\begin{equation} \label{jui}
\frac{\p g_{(x,v)\, \mu \nu }}{\p v^\alpha} \, v^\nu=0, \qquad
\frac{\p g_{(x,v)\, \mu \nu }}{\p v^\alpha} \, v^\alpha=0.
\end{equation}
It is called reversible if $g_{(x,v)}=g_{(x,- v)}$.
\end{definition}
 Sometimes we
shall write $g$ or $g_v$ for $g_{(x,v)}$ either in order to shorten the
notation or because we regard $v$ as an element of $TM\backslash 0$.
In the pseudo-Riemannian case $g$ is independent of $v$ and written without the $v$ index.
\begin{remark}
The latter equation in display is equivalent to the homogeneity
condition: for every $s>0$, $g_{(x,sv)}=g_{(x,v)}$. Thus it implies that $L:TM\to \mathbb{R}$ defined by
\begin{align}
L(x,v)&:=\frac{1}{2}\, g_{(x,v)}(v,v), \ \textrm{for} \ v\ne 0,  \label{deg}\\
L(x,0)&:=0
\end{align}
is positively homogeneous of second degree, namely for every $s>0$,
and $v\ne 0$, $L(x,sv)=s^2 L(x,v)$. The former equation implies
\begin{align}
\frac{\p L}{\p v^\mu}(x,v)&=g_{(x,v)\, \mu \nu } v^\nu,
\label{njr}\\
\frac{\p^2 L}{\p v^\mu\p v^\nu} &=g_{(x,v)\, \mu \nu}. \label{kip}
\end{align}
\end{remark}

We could have defined the pseudo-Finsler manifold as a pair $(M,L)$
in which $L$ is positive homogeneous of second  degree, and where
$g$ is defined through Eq.\ (\ref{kip}). This is the definition
adopted by most authors. Indeed, differentiating twice with respect to $s$,
$L(x,sv)=s^2L(x,v)$, and setting $s=1$ gives Eqs.\ (\ref{deg}). Equation (\ref{njr}) is obtained from  Eq.\ (\ref{kip}) observing that $\frac{\p L}{\p v^\mu}$ is positively homogeneous of first degree.

Our definition dispenses with  additional differentiability
conditions that would have to be imposed on $L$ in order to define
$g$.
Furthermore, it has the advantage of making clear
the
 connection with pseudo-Riemannian
geometry. Also it clarifies that not every tensor $g_{(x,v)}$,
positive homogeneous of zero degree in $v$, is a pseudo-Finsler
metric as the first equation in (\ref{jui}) has to be satisfied.


A geodesic is a stationary point of the functional (with a prime we
denote differentiation, typically with respect to a parameter $s$,
if the parameter is $t$ we often use  a dot)
\[
S[x]=\int_{s_0}^{s_1} L(x,x')\, \dd s , \qquad x:[s_0,s_1]\to M,\
x(s_0)=x_0, \ x(s_1)=x_1 ,
\]
where $x\in C^1([s_0,s_1])$.
By the same argument used above for $H^\mu$ we cannot demand that $g$
exists and is continuous on the zero section unless $L$ is
quadratic in the velocities, which corresponds to the case of
pseudo-Riemannian geometry.  In our terminology the Finsler
(Riemannian) structures are special cases of the pseudo-Finsler
(pseudo-Riemannian) ones. In the former case $\sqrt{2L}$ is often
denoted $F$.

We observe that it is always possible to introduce an auxiliary
Riemannian metric $h$ over $M$ and to consider the unit sphere
subbundle of $TM$. If $x$ is kept fixed then $g_{(x,v)\, \mu \nu}$
depends only on the direction and orientation of $v$ namely on $\hat{v}$,
and since the unit sphere bundle is compact over compact subsets,
$g_{(x,v)\, \mu \nu}$ is  bounded in relatively compact
neighborhoods of points of $TM$ belonging to the zero section. The
same observation holds for the partial derivatives with respect to
$x$ of $g$, and hence combinations such as $g_{v\, \alpha \beta}
v^\delta$ or $\frac{\p g_{v\, \nu \beta}}{\p x^\alpha} v^\delta$ are
locally Lipschitz also at the zero section
 once they are defined to vanish there. In particular, Eqs.\
 (\ref{deg}) and (\ref{njr}) make sense also for $v=0$, and for fixed $x$, $L$ is $C^{1,1}$ on the zero section and $C^{3,1}$ outside it.

The Lagrangian $L$ is constant over the geodesics because, using the
Euler-Lagrange equations (we cannot invoke the Hamiltonian to obtain
this result since we have not proved the convexity of $L$ in the
velocities)
\[
\frac{\dd L}{\dd t}=\frac{\p L}{\p x^\mu}\, v^\mu+ \frac{\p L}{\p
v^\mu} \frac{\dd v^\mu}{\dd t}=(\frac{\dd }{\dd t}\frac{\p L}{\p
v^\mu}) v^\mu+ \frac{\p L}{\p v^\mu} \frac{\dd v^\mu}{\dd
t}=\frac{\dd }{\dd t} (\frac{\p L}{\p v^\mu} v^\mu)=2\frac{\dd
L}{\dd t}
\]
where in the last step we used the homogeneity of $L$. The spray reads
\[
H^\mu(x,v)=-\frac{1}{2}g_v^{\mu \nu}\,(\frac{\p g_{v\, \nu
\alpha}}{\p x^\beta}+\frac{\p g_{v\, \nu \beta}}{\p
x^\alpha}-\frac{\p g_{v\, \alpha \beta}}{\p x^\nu} ) \, v^\alpha
v^\beta,
\]
where it is understood that $H^\mu(x,0)=0$. It is Lipschitz as
required because $g^{-1}_{(x,v)}=g^{-1}_{(x,\hat{v})}$ depends
continuously on the unit sphere bundle which is compact over compact
subsets of $M$, thus the inverse $g^{\mu \nu}_{(x,v)}$ stays bounded
in relatively compact neighborhoods of points belonging to the zero
section, and combinations of the form $ g^{\mu \nu}_{(x,v)} v^\beta$
are locally Lipschitz everywhere once they are  defined to vanish on
the zero section.

We shall return to the geometry of pseudo-Finsler spaces when we
 discuss Gauss' Lemma. Any mention to the various  connections
that can be introduced in this theory will be avoided in both
results and
 proofs.


\subsection{The exponential map for sprays}

As we mentioned, the non-stationary solutions to Eqs.\
(\ref{fhs})-(\ref{fha}) will be called {\em geodesics}. As this is a
system of  first order ODE over $TM$, according to the
Picard-Lindel\"of theorem, the existence and uniqueness of its
solutions are guaranteed by the locally Lipschitz condition on
$H^\mu$.

Let $\gamma_v(t)$ be the unique geodesic which starts from $\pi(v)$
with velocity $v$. The set $\Omega$ is given by those $v$ for which
the geodesic exists at least for $t\in [0,1]$. The exponential map
$\exp\colon  \Omega \to M\times M$ is given by
\[
v \mapsto (\pi(v),\gamma_v(1)),
\]
while the pointed exponential map at $p \in M$, is $\exp_p\colon
\Omega_p \to M$, $\Omega_p=\Omega\cap \pi^{-1}(p)$, $\exp_p
v:=\gamma_v(1) =\pi_2(\exp v) $. By the homogeneity of $H^\mu$ on
velocities we have
\begin{equation} \label{onu}
\gamma_{sv}(t)=\gamma_v(st),
\end{equation}
thus the set $\Omega$ (and $\Omega_p$) is star-shaped in the sense
that if $v\in \Omega$ then $s v\in \Omega$ for every $s\in [0,1]$.
Equation (\ref{onu}) clarifies that it make sense to call {\em affine}  the
geodesic parameter, for any affine
reparametrization of a geodesic gives a curve which solves the geodesic equation.

\begin{remark} \label{hyb}
The exponential map of the reverse spray, denoted $\tilde{exp}$ is
\[\tilde{exp} \,v:= (\pi(v),\gamma_{v}(-1)), \qquad \tilde{exp}_p\,
v:=\gamma_{v}(-1),\] and since in general $\gamma_v(-1)\ne
\gamma_{-v}(1)$ this map cannot be simply expressed through the
exponential map $\exp$. Of course, if the spray is reversible it
coincides with $v \mapsto \exp(-v)$.
\end{remark}


%
%

Hartman \cite{hartman50} proved that for connections the uniqueness
of the geodesic equation is lost if the Lipschitz condition
 is weakened to continuity \cite{hartman50}. This
result was improved by Hartman and Wintner \cite{hartman51b}
\cite[Exercise 6.2, Chap.\ 5]{hartman64} who considered the metric
\[
\dd s^2=(1+\vert v\vert^{1+\alpha})(\dd u^2+ \dd v^2)
\]
for $0<\alpha<1$. Its connection satisfies an H\"older
condition of exponent $\alpha$, and on any neighborhood of $p=(0,0)$
one can find infinite geodesics which start from $p$ with velocity
$(1,0)$.

These examples suggest the Lipschitz condition as the best
differentiability condition  that can be placed on a
spray.

\begin{remark}
Actually, if the  connection is that of a Riemannian $C^2$ surface
of Euclidean 3-space then uniqueness of geodesics is guaranteed, and one can even
build $C^1$ normal coordinates even though the connection is just
continuous \cite{hartman50}.  Moreover, still in the 2-dimensional
case under Lipschitzness of the connection one can prove results
which are stronger than those considered in this
work\footnote{Please notice that according to \cite{hartman83} claim
III in \cite{hartman51b} is incorrect.}
\cite{hartman50,hartman51b,hartman51c}. Nevertheless, we shall work
in the general $n$-dimensional case since the 2-dimensional one
appears too special and less relevant for applications (it suffices
to recall that in 2-dimensions any metric is locally conformally flat).
\end{remark}

In this work we shall prove that the exponential map of every spray
is a local Lipeomorphism (local bi-Lipschitz homeomorphism) and that
on $M$ any point admits a topological base of convex neighborhoods
(Theor.\ \ref{nsx}).

In a  Riemannian framework, this result can be  improved in some
directions. For instance, it is well known that Riemannian spaces
with sectional curvature bounded from below or above find a
remarkable generalization in the notion of Alexandrov spaces. In this
quite general setting there are indeed results on the existence of
convex neighborhoods \cite[Prop.\ 5.5]{berestovskij93}
\cite{perelman93,perelman93b}.

We  were finishing this work when we learned that Kunzinger,
Steinbauer and Stojkov\'ic, in a recent preprint \cite{kunzinger13},
have also provided a  proof of the bi-Lipschitzness of the pointed
exponential map and of the existence of convex neighborhoods. As with us,
they were motivated by a recent work by  Chrusciel and Grant on
causality theory under low differentiability conditions
\cite{chrusciel12}. Their approach is complementary to our own and
deserves some comments. They consider a net of smooth Riemannian
metrics $g_\epsilon$ obtained from  $g$ through convolution with a
mollifier, and use methods from comparison geometry to obtain
sufficiently strong estimates on the exponential maps of the
regularized metrics, so as to be able to carry over the bi-Lipschitz
property through the limit. In order to perform this last step in
the general pseudo-Riemannian case they use some results on
comparison geometry for indefinite metrics  recently  obtained by
Chen and  LeFloch \cite{chen08}. They also show that the Riemannian
case can be dealt with using the Rauch comparison theorem.

Our approach has  several advantages among which is that of being
tailored to the  results on convexity that we wish to prove. The
differences between our strategy and more classical approaches based
on the smooth category  and  the inverse function theorem are
minimal; there is no use of comparison geometry, nor
is  regularization required. No prior knowledge of Riemannian geometry is actually
needed, for we never use the concept of (sectional) curvature or
Jacobi field. Our results are therefore obtained by improving some
local analytical results, without introducing advanced topics in
differential geometry or touching the very foundations of the theory
under consideration. This is desirable since we are actually
obtaining basic results on local convexity which could be placed at
the very beginning of treatments on differential geometry under low
regularity.
 Our study may be useful, for instance, to understand the limits
of pseudo-Finsler geometry and particularly Lorentzian geometry, for
which a theory of the same generality of Alexandrov's is
missing.\footnote{There are well known difficulties in this
generalization. They are related to the fact that  sectional
curvature bounds imply constant  curvature \cite{harris82,nomizu83}.
These problems could be sidestepped imposing only bounds on the
sectional curvature of timelike planes \cite{harris82}.}

We have also tried to be as a complete as possible. In this way the
reader will be able to refer to the results of this work without the
need of making adjustments  in the attempt of extending the
results herein obtained. For instance, we prove that the {\em
non-pointed} exponential map $\exp$ is a Lipeomorphism from a
neighborhood of the zero section to a neighborhood of the diagonal
on $M\times M$. This result is quite useful in applications, for
instance  in causality theory  it is used in the proof that the
causal relation over  convex normal sets is closed (Theor.\
\ref{bik}).

Unless otherwise specified, $\Vert \, \Vert$ will denote the
Euclidean norm on $\mathbb{R}^n$. Let us recall that a function
$f\colon O \to \mathbb{R}^k$ defined on an open set $O\subset
\mathbb{R}^n$ is Lipschitz if, for some $K>0$ and for every $p,q\in O$,
\[
\Vert g(p)- g(q)\Vert< K \Vert p-q\Vert .
\]
It is {\em locally Lipschitz} if this inequality
holds over every compact subset of $O$, with $K$ dependent on the
compact subset.
 If $f$ depends
on another variable $z$, then $f$ is uniformly Lipschitz if the
Lipschitz constant does not depend on $z$. It is {\em locally
uniformly Lipschitz} if, chosen any compact set on the domain of
$z$, the Lipschitz constant can be chosen to be dependent on just
the compact set rather than $z$. As a consequence, a function of,
say, two variables $f(x,y)$ which is Lipschitz is also
 Lipschitz in one variable uniformly in the other. A  {\em
Lipeomorphism} $f$ is a homeomorphism for which both $f$ and $f^{-1}$ are locally Lipschitz. In the cases that will interest us $f$ will be defined in an open subset $O\subset \mathbb{R}^n$, and $f(O)$ will also be an open subset of $\mathbb{R}^n$.

%
%
%

It is well known  that the ODE $\dot{x}=f(x)$ for which $f$ is
Lipschitz admits unique solutions which have a Lipschitz dependence
on the initial conditions \cite[Ex.\ 1.2, Chap.\ 2]{hartman64}
\cite[Prop.\ 1.10.1]{cartan71} \cite[Cor.\ 1.6]{lang95}. As a
result, the exponential map $\exp$ and its pointed version $\exp_p$
are locally Lipschitz.

We shall improve this  result as in the next theorem. This
refinement will be used in the proof that in a Riemannian
 space the geodesics are locally length minimizing
in  the family of absolutely continuous  curves (and to prove an
analogous result in the Lorentzian case). To increase readability we
 postpone most proofs to the next sections.

\begin{theorem} \label{flo}
Let us consider a Lipschitz spray (Lipschitz $H^\mu$) on a $C^{2,1}$
manifold $M$, and let $\varphi\colon W\to TM$, $W\subset
[0,+\infty)\times  TM$,  $\varphi(t,v):=\gamma_v'(t)$ be the
geodesic flow map defined at those $(t,v)\in \mathbb{R} \times TM$
for which the geodesic $\gamma_v$ extends up to time $t$ (so that
the expression on the right-hand side makes sense).  Then  $W$ is an
open subset of $[0,+\infty)\times  TM$ such that $[0,1]\times \Omega
\subset W$ and such that for every $s\in [0,1]$ if $(t,v) \in W$
then $(st,v),(t,sv)\in W$. Analogously, $\Omega\subset TM$ is open
and star-shaped in the sense that for every $s\in [0,1]$ if $v\in
\Omega$ then $sv\in \Omega$.

Moreover, $\varphi(\cdot, v)$ is $C^{1,1}$, $\varphi$ is locally
Lipschitz and there is a star-shaped subset $\tilde{\Omega}\subset
\Omega$ such that, $\Omega\backslash \tilde{\Omega}$ has zero
Lebesgue measure, and for every $v\in \tilde{\Omega}$ and   every
$t\in [0,1]$, $\varphi(t,\cdot)$ is differentiable at $(t,v)$ and
the differential (Jacobian) $\p_2 \varphi(\cdot, v)$ is locally
Lipschitz in $t$, locally uniformly with respect to those $v$
belonging to $\tilde{\Omega}$ (that is the local Lipschitz constant
can be chosen so that it does not vary in a small neighborhood of
$v$ as long as the independent variable stays in $\tilde\Omega$).
Finally, for any $v\in \tilde{\Omega}$ we have that for almost every
$t\in [0,1]$ the following mixed differentials exist, are locally
bounded  and coincide
\[
\p_1 \p_2 \varphi = \p_2 \p_1 \varphi .
\]
These conclusions do not change if we restrict $\varphi$ to some
Lipschitz $m$-dimensional submanifold $N$ of $TM$.
In this case the differential $\p_2$ refers to the variables of the
Lipschitz chart on $N$ and the almost everywhere existence of $\p_2
\varphi$  must be understood in the $m$-dimensional Lebesgue measure
of $ N$.
\end{theorem}

One would like to prove  that the (pointed) exponential map is
invertible, say a local Lipeomorphism or a local diffeomorphism.
Hartman \cite[Exercise 6.2, Chap.\ 5]{hartman64} \cite{hartman83}
showed in the connection case that the exponential map is actually
$C^1$ and hence a local diffeomorphism provided the connection
admits a continuous exterior derivative. We shall not impose these
additional conditions.


The local injectivity of the exponential map for Lipschitz
connections was  proved by Whitehead in his paper on the existence
of convex normal neighborhoods \cite{whitehead32}. In
\cite{whitehead33} he observed that the result could be generalized
to sprays. Whitehead uses a theorem by Picard which applies to
boundary value problems of second order ODEs \cite[Cor.\ 4.1, Chap.\
12]{hartman64}. Using the theorem on the invariance of domain one
could infer that the exponential map provides a local homeomorphism,
though from here it does not seems easy to obtain that it admits a
Lipschitz inverse.


Our approach relies instead on an improved version of the inverse
function theorem due to Leach \cite{leach61} (see also
\cite{cartan71}). This theorem depends on Peano's definition of
strong derivative \cite{peano92} and on the natural corresponding
notion of strong differential studied by Severi. Unfortunately,
Peano's contributions in this direction are, together with many
other accomplishments by the Italian mathematician, little known
\cite{dolecki12}. Nijenhuis' attempt \cite{nijenhuis74} to
popularize Peano's choice and Leach's inversion theorem  passed
essentially unnoticed. Peano's choice provides a better definition
of differential, so good, in fact, that having to choose one  should
probably adopt it in place of  the usual differential in analysis
textbooks. Indeed, the strong differential  leads to stronger and
more elegant results, and seems to corresponds better with
intuition.
 We hope that this study, showing
the usefulness of Peano's strong derivative for the exponential map
will serve to motivate its mention in University courses.

Let us recall Peano's definition of strong differential and its
basic properties. We give a general definition for Banach spaces
although we  shall work on  $\mathbb{R}^k$ for some $k\ge 1$. We
denote with $B(p,r):=\{q: \vert q-p\vert<r \}$ the open ball of
radius $r$ centered at $p$, and with $\bar{B}(p,r):=\{q: \vert
q-p\vert\le r \}$ the closed ball.

\begin{definition}
Let $E$ and $F$ be Banach spaces,  and let $f\colon  O \to F$, be a
function defined on an open set $O\subset E$. The strong
differential of $f$ at $p \in O$ is a bounded linear transformation
$L\colon E \to F$ which approximates changes of $f$ in the sense
that for every $\epsilon > 0$, there is a $\delta >0$ such that if
$\vert q_1-p\vert \le \delta$ and $\vert q_2-p\vert \le \delta$,
then:
\begin{equation} \label{nis}
\vert f(q_1) -f(q_2) - L(q_1 - q_2) \vert \le \epsilon \vert q_1 -
q_2\vert .
\end{equation}
\end{definition}

Clearly, if the strong differential at $p$ exists then it is unique.
If $f$ is strongly differentiable at $p$ then taking $q_2=p$ shows
that it is also Fr\'echet differentiable and that the differentials
so defined coincide.  In the finite dimensional case which will
interest us all norms are equivalent thus the notions of strong
differentiation and  strong differential are independent of the norm
used.

We list some properties of strong differentiation which are easy to
prove \cite{peano92,leach61,leach63,esser64,nijenhuis74}.
\begin{itemize}
\item[(i)] If $f$ is strongly differentiable at $p$ then it satisfies a
Lipschitz condition in a neighborhood of $p$.\footnote{One could ask
whether every Lipschitz function is strongly differentiable almost
everywhere. The answer is negative already for functions defined on
the real line \cite[Sect.\ 14.4.1]{garg98}.}
\item[(ii)] If $f$ is differentiable in a neighborhood of $p$ and the
differential is continuous at $p$ then it is strongly differentiable
at $p$.
\item[(iii)] If $f$ is strongly differentiable over a subset $A\subset E$ then
the strong differential is continuous over $A$ with respect to the
induced topology.
\item[(iv)] If a continuous function $f:U_1\times U_2\to V$ on a product Banach
space  admits strong partial differentials at $p$ (obtained keeping
constant the other variable) then it is strongly differentiable at
$p$ and the total differential has the usual gradient expression in
terms of partial differentials.
\item[(v)] The composition of strongly differentiable functions is strongly differentiable.
\item[(vi)] The mixed partial strong derivatives coincide wherever
they exist \cite{minguzzi13b}.
\item[(vii)] If $f\colon \mathbb{R}\to  \mathbb{R}$ has positive strong
differential at $p$, then $f$ is continuous and increasing in a
neighborhood of $p$.
\end{itemize}

Conditions (ii) and (iii) clarify that a function is $C^1$ over an
open set if and only if it is strongly differentiable over it. Some
key theorems in analysis that require a $C^1$ condition on a
neighborhood of $p$ can be  proved demanding the weaker condition of
strong differentiability at $p$. An example is Leach's inversion
theorem \cite{leach61,cartan71,nijenhuis74,krantz00} which
generalizes Dini's and which we state in a  form suitable for our
purposes:

\begin{theorem}[Leach]
Let $f\colon O \to \mathbb{R}^n$, be a function  defined on an open
subset $O\subset \mathbb{R}^n$,   such that $f$ has strong
differential $L\colon \mathbb{R}^n\to\mathbb{R}^n$ at $p\in O$. If
$L$ is invertible then there are an open neighborhood $N_1$ of $p$,
an open neighborhood $N_2$ of $f(p)$, and a function $g: N_2\to \mathbb{R}^n$
such that, $f(N_1)=N_2$,   $g(N_2)=N_1$, $f\vert_{N_1}$ and $g$ are
one the inverse of the other, they are both Lipschitz  and $g$ has
strong differential $L^{-1}$ at $f(p)$.

Moreover, in this case $f$ is differentiable at $q\in N_1$ if and
only if $g$ is differentiable at $f(q)$,  in which case the
differentials are invertible. This last statement holds also with
{\em differentiable} replaced by {\em strongly
differentiable}.\footnote{The fact that $N_1$ and $N_2$ can be
chosen to be open sets such that both $f\vert_{N_1}$ and $g$ are
Lipschitz follows from Leach's original formulation plus (i). The
statement in the last paragraph is not contained in Leach's original
formulation but can be found in its proof.}
\end{theorem}


In order to clarify the connection between this inversion theorem
and Clarke's \cite{clarke76b} it is convenient to recall the notion
of Clarke's generalized differential for locally Lipschitz
functions:

\begin{definition}
The generalized Jacobian of a locally Lipschitz function $f\colon O
\to \mathbb{R}^n$, $O\subset \mathbb{R}^k$ at $p$, denoted $\p
f(p)$, is the convex hull of all matrices $M$ of the form
\[
M = \lim_{p_i\to p}  df(p_i)
\]
where $p_i$  converges to $p$, $f$ is differentiable at $p_i$ for
each $i$ and $d f$ denotes the usual Jacobian.
\end{definition}


By Rademacher's theorem the generalized differential is non-empty at
$p$ and we have (see also \cite{esser64,clarke76b})

\begin{proposition}
If  $f\colon O \to \mathbb{R}^n$, $O\subset \mathbb{R}^k$, is
strongly differentiable at $p$ then $\p f(p)=\{\dd f(p) \}$.
\end{proposition}

\begin{proof}
Indeed, take any $\epsilon>0$ and let $\delta>0$ be such that Eq.\
(\ref{nis}) holds. Let $J$ be the limit of a sequence $J_i=d f(p_i)$
for $p_i\to p$. We can assume $\Vert p_i-p\Vert<\delta/2$ for each
$i$. Let $e$ be any normalized unit vector. For each $i$ we can find
some $0<\delta_i\le \delta/2$ such that
\[
\Vert f(q_i)-f(p_i)-d f(p_i)(q_i-p_i)\Vert \le \epsilon \Vert
q_i-p_i\Vert
\]
for every $q_i$ such that $\Vert q_i-p_i\Vert\le\delta_i$. Let
$q_i=p_i+\delta_i e$, then
\begin{align*}
\delta_i \Vert (L-J_i) e\Vert &=\Vert (L-J_i)(q_i-p_i)\Vert\le \Vert
f(q_i)-f(p_i)-J_i(q_i-p_i)\Vert\\
&\quad + \Vert f(q_i)-f(p_i)-L(q_i-p_i)\Vert \le 2\epsilon \Vert
q_i-p_i\Vert=2\epsilon \delta_i.
\end{align*}
Simplifying $\delta_i$, taking the limit $i\to \infty$ and using the
arbitrariness of $\epsilon$ and $e$ proves that $J=d f(p)$ and hence
$\p f(p)=\{d f(p)\}$.
\end{proof}

Clarke proved that if  $k=n$ and $\p f$ admits only invertible
elements then $f$ is a local Lipeomorphism.  Leach's version states
something more for $f$ strongly differentiable at $p$, for  it
establishes that the inverse is strongly differentiable at $f(p)$.

Our strategy is then clear: we are going to prove the strong
differentiability of the exponential map in order to deduce  the
Lipschitzness of the inverse by means of Leach's (or Clarke's) inversion
theorem. The proof of the strong differentiability of the
exponential map will pass through a local analysis based on   the
Picard-Lindel\"of approximation method.

In the end we shall prove:

\begin{theorem} \label{jfd} Let  $M$ be a $C^{2,1}$-manifold endowed with a locally Lipschitz
spray.
\begin{itemize}
\item[$(\exp)$]
The set $\Omega$ is open in the topology of $TM$. The exponential
map $\exp\colon \Omega \to M\times M$, $\Omega \subset TM$, is
locally Lipschitz.


Moreover, $\exp$ is strongly differentiable
over the zero section, namely over the image of $p\mapsto 0_p$.
The map $\exp$ provides a Lipeomorphism between an open star-shaped
neighborhood of the zero section and an open neighborhood of the
diagonal of $M\times M$.
\item[$(\exp_p)$]  For every $p\in M$ the set $\Omega_p$ is open in the topology of
$T_pM$. The pointed exponential map  $\exp_p\colon \Omega_p \to M$,
$\Omega \subset T_pM$, is locally Lipschitz and strongly
differentiable at the origin. The map $\exp_p$ provides a local
Liperomorphism from a star-shaped open subset of $\Omega_p$ and an
open neighborhood of $p$ (for more see Theorems \ref{flo} and
\ref{nsx}).



\end{itemize}
\end{theorem}

In the pointed case this result can be refined. We shall need some
definitions.

\begin{definition}
An open neighborhood $N$ of $p\in M$ will be called 
{\em
normal} if there is an open star-shaped subset $N_p\subset \Omega_p$
such that $\exp_p\colon N_p \to {N}$ is a Lipeomorphism.
\end{definition}


\begin{definition}
An open set $C\subset M$ will be called {\em convex normal} if it is
a normal neighborhood of each of its points. We shall say that
$\bar{C}$ is {\em strictly convex normal} if $C$ is convex normal
and any two points of $\bar{C}$ are connected by a unique geodesic
contained in $C$ but for the endpoints. A (strictly) convex normal
set is caller {\em reversible} if it is so also for the reverse
spray.
\end{definition}

\begin{remark}
In a reversible convex normal subset $C$ for  any two points $p,
q\in C$ there is a geodesic $\gamma_{pq}\colon [0,1]\to C$
connecting $p$ to $q$ and a geodesic $\gamma_{qp}\colon [0,1]\to C$
connecting $q$ to $p$. They coincide if the spray is reversible.
Furthermore, there are geodesics for the reverse spray
$\tilde\gamma_{pq}(t)=\gamma_{qp}(1-t)$ connecting $p$ to $q$ and
$\tilde \gamma_{qp}(t)=\gamma_{pq}(1-t)$ connecting $q$ to $p$. The
last identities follow from the uniqueness of the connecting
geodesic for the  spray. Observe that while the convexity of $C$
with respect to the spray implies the convexity of $C$ with respect
to the reverse spray, a {\em reversible convex normal set} has a
stronger property which cannot be deduced from the corresponding
property for the spray, that is, in a  {\em reversible convex normal
set} $C$ we have that $\tilde\exp_p^{-1}\vert_C$ is a {\em
Lipeomorphism} for every $p \in C$.
\end{remark}
The following concept will be useful in the next section.
\begin{definition}
Let $C$ be a convex normal set, let $p,q\in C$ and let $x: [0,1]\to
C$, $x(0)=p$, $x(1)=q$, be the unique geodesic connecting them. The
vector $\dot{x}(1)$ is denoted $P(p,q)$ and called {\em position
vector}.
\end{definition}

We shall prove:

\begin{theorem} \label{nsx}
Let  $M$ be a $C^{2,1}$-manifold endowed with a locally Lipschitz
spray. Let $O$ be an open neighborhood of $p\in M$.
 Then there is a reversible strictly convex  normal neighborhood $C$ of $p$
contained in $O$, such that $\exp$ establishes a Lipeomorphism
between an open star-shaped subset of $TC$ and $C\times C$.
Analogously, $\tilde{exp}$ establishes a Lipeomorphism between an
open star-shaped subset of $TC$ and $C\times C$.

Moreover, for every chart $\{x^\mu\}$ defined in a neighborhood of
$p$, $C$ can be chosen equal to the open
ball $B(p,\delta)$ for any sufficiently small $\delta$  (the ball is
defined through  the Euclidean norm induced by the coordinates).
\end{theorem}

If the spray  is compatible with a pseudo-Finsler structure then
this result can be further refined.

\begin{remark} \label{djg}
The previous theorems can be formulated in more generality for
$C^{k,1}$ sprays over $C^{k+2,1}$ manifolds with $k\ge 0$ or for
$C^{k+1, \alpha}$ sprays over $C^{k+3,\alpha}$, $\alpha \in [0,1)$,
manifolds. The exponential map and its inverse have the   degree of
differentiability of the spray.

These cases follow more or less straightforwardly  from the
Lipschitz $k=0$ case treated here or from what is already known for
$C^1$ sprays. For a direct  proof that generalizes that for the
Lipschitz case see Remark \ref{djh}.
\end{remark}

\subsection{Gauss' Lemma for pseudo-Finsler sprays}


Let us consider again a pseudo-Finsler geometry in which the
fundamental tensor $g_v$ is $C^{1,1}$ and the spray is Lipschitz.
This means that in the pseudo-Riemannian case $g$ is $C^{1,1}$ and
the connection is Lipschitz.

The local length minimization property of geodesics in Riemannian
spaces, or the local Lorentzian length maximization property  of
causal geodesics in Lorentzian manifolds, are proved passing through
Gauss' Lemma  (in the Lorentzian case see  \cite{oneill83,chrusciel11}).
This Lemma is known to hold in Finsler geometry
\cite{bao00}

\begin{quote}(Gauss' Lemma under sufficient differentiability conditions)\\
Let $p\in M$, let $N$ be a normal neighborhood of $p$ and let $v\in
\exp_p^{-1} N\backslash 0$. Let $w\in T_p M\sim T_v(T_p M)$. Then
\begin{equation} \label{juw}
g_{(d \exp_p)_v v}( (d \exp_p)_v v, (d \exp_p)_v w )=g_v(v,w).
\end{equation}
\end{quote}


This lemma is usually expressed as above using the push forward of
the exponential map \cite[Theor.\ 3.70]{gallot87} \cite{docarmo92}.
As the exponential map is $C^1$ for $C^2$ metrics, one expects this
lemma to be valid for $C^2$ metrics.

The just mentioned classical proofs of Gauss' lemma work indeed in the $C^2$ case. Without entering in too many  details the reader should just keep in mind that for what concerns differentiability with respect to the initial conditions the exponential map, by Peano's theorem \cite[Theor.\ 3.1]{hartman64},  behaves better in the radial direction than in the transverse directions. As a consequence, the mixed derivative of the expression $f(t,s)=\exp_p (t v(s))$ is continuous if $v(s)$ is $C^1$, cf.\ \cite[Cor.\ 3.2]{hartman64}. Thus one can use Schwarz theorem and  fully justify the proof of \cite{docarmo92}.

Other proofs seem less convincing \cite[Cor.\ 2.2]{senovilla97}. It is very easy to forget that one cannot work directly with, say, derivatives of vector or  tensor fields expressed in a
normal coordinate chart or the Christoffel symbols in normal
coordinates,
indeed, since this chart is
just $C^1$, vector and tensor fields on them can be at most $C^0$
and hence are not differentiable and similarly the Christoffel symbols
are not defined \cite{deturck81}.




The situation is  worse for Lipschitz connection or sprays and
$C^{1,1}$ metrics.  Under this differentiability hypothesis the exponential map is just Lipschitz, thus
 Gauss' Lemma is not even expected to hold. However, we
shall show that Gauss' Lemma still holds true in a formulation which
does not require the differentiability of the exponential map.

We shall need some  technical result either on (a) the
differentiation under the integral sign, or on (b) Schwarz's theorem
on the equality of mixed partial derivatives. We shall discuss them
in Section \ref{jtt}. In the end we shall prove (in order to obtain
a more restrictive  pseudo-Riemannian version just remove the vector
index from the fundamental tensor $g$):


\begin{theorem} \label{gux}
Let  $(M,g)$ be a $C^{2,1}$ pseudo-Finsler manifold for which $g$ is
$C^{1,1}$ (Sect.\ \ref{hbh}).
 Let $N$ be a normal neighborhood of
$p\in M$.
The function $D^2_p\colon  N \to \mathbb{R}$ defined by
\begin{equation} \label{jeg}
D^2_p(q):= 2L(p, exp_p^{-1}(q))=\,g_{exp_p^{-1}(q)}(exp_p^{-1}(q),
exp_p^{-1}(q))
\end{equation}
is $C^{1,1}$ in $q$  and
\begin{equation} \label{jwf}
d D^2_p(q)= 2 g_{P(p,q)} (P(p,q),\cdot),
\end{equation}
where $P(p,q)= \gamma'_{\exp^{-1}_p q}(1)$ is the position vector of $q$ with respect to $p$.
Thus the level sets of $D^2_p$ are orthogonal to the geodesics
issued from $p$, and for $t,s>0$ the $(-t)$-time flow map of
$P(p,\cdot)$ is a Lipeomorphism between $(D^2_p)^{-1}(s)$ and its
image on $(D^2_p)^{-1}(s\, e^{-2t})$.

Finally, Eq.\ (\ref{juw}) holds wherever $\exp_p\colon \exp_p^{-1} N
\to N$ is differentiable, hence almost everywhere. Thus the usual
Gauss' Lemma holds under $C^2$ differentiability of the metric $g$.
\end{theorem}


Geometrically Eq.\ (\ref{jwf}) states the geodesic  connecting $p$
to $q$ is perpendicular to a level set $D^2_p=cnst$ (in
pseudo-Finsler geometry $v$ is perpendicular to $w$ if
$g_v(v,w)=0$).

 Observe that $D^2_p$ can be negative if $g$ is not
positive definite. If $g$ is positive definite, as we shall prove in
a moment in Theorem \ref{pwa}, $D_p(q)$ coincides with the Finsler
distance  between $p$ and $q$ in the space $(N,g\vert_N)$ and hence
coincides with the Finsler distance on $M$ provided  the ball of
radius $D_p(q)$ centered at $p$ is contained in $N$.

\begin{remark}
It is  somewhat surprising that $D^2_p$ is $C^{1,1}$. One would
expect it to be Lipschitz because from its definition, that is Eq.\
(\ref{jeg}), we see that it is built from the inverse of the
exponential map which is Lipschitz. The additional degree of
differentiability comes from the fact that we can check that the
differential is almost everywhere as in Eq.\  (\ref{jwf}). This
result can then be extended everywhere thanks to the Lipschitzness
of $D^2_p$ and the continuity of $P(p,q)$ (see Theorem \ref{pok}).
\end{remark}

\begin{remark}
The proof of this result can be easily generalized to show that for $(p,q)$ belonging to a reversible convex normal set, $D_p(q)$ is $C^{1,1}$ in $(p,q)$ and its differential is
\[
d D^2_p(q) (v_p,v_q)= 2 g_{P(p,q)} (P(p,q),v_q)+2 g_{\tilde{P}(q,p)} (\tilde{P}(q,p),v_p),
\]
where $v_p\in T_pM$, $v_q\in T_qM$, and $\tilde{P}$ is the position vector map according to the reverse spray. Furthermore, with Prop.\ \ref{juy} we shall prove  that $P(p,q)$ is strongly differentiable on the diagonal of $M\times M$ hence $D^2_p(q)$ has first differential which is strongly differentiable at the origin.
\end{remark}

We shall state the next result for pseudo-Finsler manifolds for
which the fundamental tensor is either positive definite (Finsler
geometry) or of signature $(-,+,\cdots,+)$ (Lorentzian-Finsler
geometry). It is necessary to elaborate the last structure in the
notion of {\em Finsler spacetime} which extends the usual notion of
spacetime met in mathematical relativity.

Let us start from a Lorentzian-Finsler manifold $(M,g)$, and let us
keep in mind that if the fundamental tensor $g_v$ does not depend on
the velocity then we are back to a Lorentzian manifold
\cite{beem96}.  Non vanishing vectors are called {\em spacelike},
{\em lightlike} or {\em timelike} depending on the sign of
$g_v(v,v)$, namely positive, null or negative, and the terminology
extends to $C^1$ curves provided the causality type of the tangent
vector is consistent throughout the curve (which is assured for
geodesics since in that case $g_v(v,v)$ is constant over the curve). A vector is {\em null} if it is lightlike or zero, and {\em non-spacelike} if it is causal or zero.
 At any $x\in M$ let us denote with $I_x\subset T_xM$ the subset of
timelike vectors, and with $J_x$ the subset of non-spacelike
vectors and with $E_x$ the subset of null vectors.

Beem and Perlick \cite{beem70,perlick06} have shown that each
component of $I_x$ is convex, and hence, by continuity, that each
component of $J_x\backslash\{0\}$ is convex. Since $\p L(x,v)/\p
v\ne 0$ for $v\ne 0$, the hypersurfaces $g_v(v,v)=cnst$ are imbedded
submanifolds and hence each component of $E_x\backslash \{0\}$ plus
$\{0\}$ is the boundary of some component of $I_x$. Analogously,
each component of $J_x\backslash \{0\}$ plus $\{0\}$  is the closure
of some component of $I_x$. Furthermore, again because $\p L(x,v)/\p
v\ne 0$ for $v\ne 0$, distinct components of $J_x\backslash  \{0\}$
do not intersect.

For simplicity we shall restrict our analysis to
\begin{definition}
 {\em Finsler
spacetimes} are pseudo-Finsler manifolds $(M,g)$ for which (a) $g$
has signature $(-,+,\cdots, +)$, (b) for one (and hence every) $x\in
M$, $I_x$ has just 2 components, (c) there exists a global
continuous timelike vector field which defines a notion of {\em
future cone}.
\end{definition}


Condition (c), can be accomplished passing to a double
covering while (b) is assured under reversibility if the spacetime has dimension larger than two \cite{minguzzi13c}.

Given the time orientation, causal vectors are either future or
past, and so the regular $C^1$ causal curves are either past
directed or future directed . The $C^1$ causal curves which we shall
consider will be future directed.

\begin{remark}
Observe that we   do not assume that $g$ is reversible. Thus we have
essentially two different distributions of light cones on spacetime
and hence two causality theories. In what follows, by mentioning
only future directed timelike curves we restrict ourselves to one of
these theories.
\end{remark}

Two points $x,y\in M$ are said to be chronologically related in a
set $S\subset M$, this being denoted $y\in I^+_S(x)$, $(x,y)\in
I^+_S$ or $x\ll_S y$, if there is a future directed $C^1$ timelike
curve from $x$ to $y$ contained in $S$. Two points are said to be
causally related, this being denoted $y\in J^+_S(x)$, $(x,y)\in
J^+_S$ or $x \le_S y$, if there is a future directed $C^1$ causal
curve from $x$ to $y$ contained in $S$ or $x=y\in S$. We write
$x<_Sy$ if $x\le_S y$ but $x\ne y$. If $S=M$ then we write simply
$\ll$, $\le$ and $<$.

A curve $\sigma\colon [a,b]\to M$ will be {\em absolutely
continuous} (an AC-curve for short) if its components in one (and
hence every) local chart are locally absolutely continuous.
Equivalently, introducing a complete Riemannian metric on $M$, and
denoting by $\rho$ the corresponding distance, $\sigma$ is
absolutely continuous if it satisfies locally the usual
definition of absolute continuity between (topological) metric
spaces.
 Since every pair of Riemannian metrics over a
compact set are Lipschitz equivalent, and $M$ is locally compact,
this definition does not depend on the metric chosen. Analogously,
we can define the concept of Lipschitz curve.

We shall say that an AC-curve $\sigma\colon [a,b]\to M$, $t\mapsto
\sigma(t)$, is future directed {\em causal} if $\dot \sigma$ is is future directed causal almost everywhere. We do not need to define a
notion of absolutely continuous timelike curve.

The Lorentzian-Finsler length of a causal AC-curve is
\[
l[\sigma]=\int_a^b\! \sqrt{-g_{\dot\sigma}(\dot\sigma,\dot\sigma)}
\,\, \dd t ,
\]
and it is finite because the integrand belongs to $L^1([a,b])$ as in
coordinates we have $\dot{\sigma}^\mu\in L^1([a,b])$,
$\sqrt{\vert\dot{\sigma}^\mu\vert} \in L^2([a,b])$.

In the Finsler case the concept of Finsler length is defined
analogously but with a plus sign inside the square root. The Finsler
distance from $p$ to $q$ is the infimum of the Finsler lengths of
the $C^1$ curves connecting $p$ to $q$. It is symmetric and hence a
true distance whenever the Finsler structure is reversible. We can
also define a Lorentzian-Finsler distance between $p$ and $q$ with
$p\le q$ as the supremum of the Lorentzian-Finsler lengths of the
causal AC-curves connecting $p$ to $q$. As in Lorentzian geometry,
it satisfies a reverse triangle inequality \cite{beem96}.


\begin{theorem} \label{pwa}
Let  $(M,g)$ be a $C^{2,1}$ pseudo-Finsler manifold for which $g$ is
$C^{1,1}$ (Sect.\ \ref{hbh}). Let $N$ be a normal neighborhood of
$p\in M$ and suppose that  $g$ is
\begin{itemize}
\item[] $\!\!\!\!\!\!\!\!\!$Finsler: \\Let $\sigma\colon [0,1]\to N$, $s\mapsto \sigma(s)$, be any AC-curve
starting from $p$, then its length is larger than that of the
(unique) geodesic connecting its endpoints, unless its image
coincides with that of  that geodesic. In this last case the Finsler
distance from $p$ provides an affine parameter $r(s)$ where the
dependence on $s$  is absolutely continuous and increasing.
\item[] $\!\!\!\!\!\!\!\!\!$Lorentzian-Finsler: \\Let $\sigma\colon [0,1]\to N$ be any future directed causal
AC-curve starting from $p$, then $\exp^{-1}(\sigma(s))$ is future directed causal
for every $s>0$, and if $\exp^{-1}(\sigma(\hat s))$ is lightlike
then $\sigma\vert_{[0,\hat{s}]}$  coincides with a future directed lightlike
geodesic segment up to parametrizations. Finally, the
Lorentzian-Finsler length of $\sigma$ is smaller than that of the
(unique)  future directed casual geodesic connecting its endpoints,  unless its image
coincides with that of that geodesic. In this last case the
affine parameter of the geodesic is absolutely continuous and increasing with
$s$.
\end{itemize}
\end{theorem}

%
%
%
%
%
%
%

\begin{remark}
Physically the Lorentzian-Finsler version proves that a motion which
is almost everywhere slower than light is also locally slower than
light.  This is the main result which allows us to develop causality
theory for Lipschitz connections in Finsler spacetimes. In
particular, $I^+_N(p)$ (or $J^+_N(p)$)  coincides with the
exponential map-image in $N$ of the future directed timelike (resp.\
causal)  cone at $p$.
\end{remark}

\begin{remark}
It is natural to ask whether locally a spacelike geodesic segment
minimizes the functional $\int \sqrt{g_{\dot{\sigma}}(\dot
\sigma,\dot \sigma)} \,\dd t$ over the $C^1$ spacelike curves
connecting the same endpoints. The answer is negative already in 1+1
Minkowski spacetime, just consider almost lightlike zig-zag curves
which approximate the geodesic. Their presence shows that the
infimum of the functional vanishes.
\end{remark}

%
%
%
%
%
%
%
%
%

\subsection{Some applications to mathematical relativity}

We recall that according to Hawking and Ellis \cite{hawking73} a
future directed continuous causal curve $x:[a,b]\to M$, is a
continuous curve such that for every open convex normal set $C$
intersecting $x$, whenever $x([t_1,t_2])\subset C$, $t_1<t_2$, the
points $x(t_1)$ and $x(t_2)$ are connected by a future directed
causal geodesic contained in $C$. This definition can be imported
word by word to the realm of Finsler spacetimes.

An interesting consequence  of Theorem \ref{pwa} is Theorem
\ref{bpo} which will provide a kind of converse of the well known
fact that continuous causal curves are Lipschitz when parametrized
with respect to the arc-length of a  Riemannian metric
\cite{penrose72} (due to the light cones that bound the curve).
Known proofs in Lorentzian geometry \cite{candela08}  work under
stronger differentiability assumption which guarantee the validity
of the usual Gauss' Lemma.

For its proof we shall need a lemma (a Lorentzian version is
\cite[Lemma 2.13]{minguzzi06c}). Given two Lorentzian-Finsler
metrics $g_1$ and $g_2$ we write $g_1<g_2$ if the causal vectors for
$g_1$ are timelike for $g_2$.

\begin{lemma} \label{gpl}
Let  $(M,g)$ be a Finsler spacetime. Let $p\in M$ then we
can find in a neighborhood $O$ of $p$ a flat Lorentzian metric
(hence independent of $v$) $g^+$, such that $g<g^+$.
\end{lemma}

\begin{proof}
Let $\{x^\mu\}$ be a coordinate chart in a neighborhood $O$ of $p\in
M$. Let $h=(\dd x^0)^2+(\dd x^1)^2+\cdots+(\dd x^{n-1})^2$ be the
usual Euclidean metric and let us consider the corresponding unit
sphere subbundle of $TM$. Let $\hat{J}_x$, $x\in O$, be the
intersection of $J_x$ with the unit sphere at $x$. Since the
components of $J_x\backslash\{0\}$ are convex in the linear
structure of $T_xM$, $\hat{J}_x$ is made of two closed disjoined
convex sets of the unit sphere at $T_xM$. We can always find a great
circle separating the two convex sets (note that the
sphere has dimension $n-1$, the circle has dimension $n-2$, thus the
terminology used is not accurate for $n\ne 3$) (to prove this use
for instance the stereographic projection from a point not belonging
to the convex sets, use the Hahn-Banach theorem, and then project
back to the sphere) and we can also rotate the coordinate system so
that the hyperplane on $T_pM$ spanned by that great circle is
orthogonal to $\p_0$. By continuity the Lorentzian metric at $p$,
$g^+=-N (\dd x^0)^2+(\dd x^1)^2+\cdots+(\dd x^{n-1})^2$, satisfies
$g<g^+$ for sufficiently large $N$. Still by continuity of $J_x$ on
$x$ (this continuity can be rigourously expressed in the Hausdorff
metric on sets, but the details will not be needed) the relation
$g<g^+$  holds in a neighborhood of $p$ which we can redefine to be
$O$.
\end{proof}

\begin{theorem} \label{bpo}
Let  $(M,g)$ be a Finsler spacetime where $M$ is a
$C^{2,1}$-manifold endowed with a $C^{1,1}$ fundamental field $g$.
Let $I$ be an interval of the real line. Every  future directed causal AC-curve
$x:I\to M$ is a future directed continuous causal curve. Every future directed continuous causal
curve $x:I\to M$ once suitably parametrized (e.g.\ with respect to
the arc-length of a Riemannian metric) becomes a future directed causal locally
Lipschitz curve.
\end{theorem}

\begin{remark}
It is not true that  every continuous causal curve is a causal
AC-curve. For instance, consider the timelike geodesic of Minkowski
spacetime which satisfies $\vec{x}=0$ and which is parametrized by
$x^0$. Consider the parametrization $t=f_s^{-1}(x^0)$  where $f_s$
is a singular monotone continuous function \cite[Ex.\
8.20]{rudin70}, so that $\dot f_s=0$ almost everywhere.
\end{remark}


\begin{proof}
It is sufficient to prove it for $I=[a,b]$. Suppose that $x$ is a future directed
causal AC-curve, let $C$ be a convex normal set intersecting $x$,
and let $t_1<t_2$ be such that $x([t_1,t_2]) \subset C$. The set $C$
is a normal neighborhood for $p:=x(t_1)$ thus by Theorem \ref{pwa}
the geodesic connecting $x(t_1)$ and $x(t_2)$ is (future directed)
causal.

Conversely, suppose that $x:I\to M$ is a future directed continuous causal curve and
let $\bar{t}\in I$. By Lemma \ref{gpl} we can find $C^{2,1}$
coordinates $x^\mu$ in a convex neighborhood $C$ of $p:=x(\bar{t})$
such that for some $N>0$, the Lorentzian metric
 $g^+=-N (\dd x^0)^2+ (\dd x^1)^2+\cdots (\dd x^{n-1})^2$ satisfies  $g<g^+$ throughout $C$.

The function $x^0(t)$ must be increasing in a neighborhood of
$\bar{t}$. Indeed, for  $t_1,t_2$ belonging to a sufficiently small neighborhood of $\bar{t}$, $x(t_1), x(t_2)\in C$. By Theorem \ref{pwa} there is a future directed causal $g$-geodesic connecting $x(t_1)$ with $x(t_2)$, which is in particular a future directed $g^+$-causal $C^1$ curve. But  $x^0$ is increasing over this type of curve since
$x^0$ is the usual time coordinate for the subset $(C,g^+)$ of Minkowski spacetime, which proves the claim.

%

Once parametrized with respect to $x^0$ the curve
becomes Lipschitz because of the condition of $g$-causality which
implies $g^+$-causality which implies $\Vert
\vec{x}(t_2)-\vec{x}(t_1)\Vert \le N \vert x^0(t_2)-x^0(t_1)\vert$.
Clearly, if $l$ is an arc-length parameter induced by the Euclidean
coordinate metric $(\dd x^0)^2+(\dd \vec{x})^2$ then $x^0(l)$ is
1-Lipschitz, and so $x(l)$ is locally Lipschitz. As all Riemannian
metrics are Lipschitz equivalent over compact sets, $x$ is locally
Lipschitz whenever parametrized with respect to Riemannian
arc-length.
\end{proof}
%

%
%

Continuous causal curves in Lorentzian geometry enjoy nice
properties under various notions of  limit \cite{minguzzi07c}.  The
proofs presented in reference \cite{minguzzi07c} hold as well under
our present weaker differentiability assumptions and in the Finsler
spacetime case. Analogously, as should be expected from the above
equivalence, the family of absolutely continuous curves  is closed
under uniform convergence, a fact quite well known since the work of
Tonelli on one-dimensional variational principles \cite{buttazzo98}.

The theorems so far proved in the Lorentzian-Finsler case are
sufficient to establish
 the validity of most results of mathematical
relativity and especially of causality theory for $C^{1,1}$ metrics
on even Finsler rather than just Lorentzian spacetimes. In
particular, the notions of chronological and causal relations do not
require modifications from the standard ones \cite{hawking73}. The
chronological relation is still open, the boundaries of the causal
and chronological futures of a point coincide, the achronal
boundaries are still Lipschitz hypersurfaces and so on.

We wish to include a result of this type to show that most proofs
can be extended word by word from the Lorentzian $C^3$ metric case
to the Lorentzian-Finslerian $C^{1,1}$ metric case. Let us recall
that a set is achronal if there is no timelike curve starting and
ending at the set.

\begin{lemma}Let  $(M,g)$ be a Finsler spacetime where $M$ is a
$C^{2,1}$-manifold endowed with a  $C^{1,1}$ fundamental
field $g$. Let $p<q$ then $(p,q)\in I^{+}$ or every continuous
causal curve connecting $p$ to $q$ is an  achronal  future directed lightlike
geodesic (up to parametrizations).
\end{lemma}

\begin{proof}
Assume $(p,q)\notin I^{+}$ and let $\gamma: [0,1]\to M$ be a future
directed continuous causal curve such that $\gamma(0)=p$ and
$\gamma(1)=q$. Since the image of $\gamma$ is compact there is a
finite covering with convex normal neighborhoods $U_i$,
$i=1,\ldots,n$. We can assume $U_i\cap U_{i+1}\ne \emptyset$ and
that there are $p_i \in \gamma\cap U_i\cap U_{i+1}$,
$i=1,\ldots,n-1$, $p_0\equiv p\in U_1$ and $p_n\equiv q\in U_n$.
Since $\gamma$ is a continuous causal curve by Theorem \ref{pwa}
$(p_i,p_{i+1}) \in J^{+}_{U_{i+1}}$ thus $p_i$ and $p_{i+1}$ are
joined by a geodesic $\eta_i$ in $U_{i+1}$ and this geodesic
coincides with $\gamma$ between the same points or it is timelike.
%

Let us show that the presence of one timelike segment $\eta_i$
implies $(p,q)\in I^{+}$. This is so because from the curve made of
geodesic segments $\eta_i$ one can construct a piecewise curve made
of timelike geodesic segments. Indeed, one starts from $\eta_i$ and
translates slightly the final point of $\eta_{i-1}$ along $\eta_i$ so
that the new connecting $\eta'_{i-1}$ becomes timelike (as the
Lorentzian-Finsler distance between the new endpoints is necessarily
positive). Analogously, one translates slightly the starting point of
$\eta_{i+1}$ along $\eta_i$ so that the new connecting $\eta'_{i+1}$
becomes timelike. Then one continues in this way by taking as
reference the timelike geodesic segments $\eta'_{i-1}$ or
$\eta'_{i+1}$. The corners of the so obtained piecewise timelike
curve can be finally smoothed out.

Also note that if all the segments $\eta_i$ are lightlike but do not
join smoothly then one can, arguing as above, replace one lightlike
segment with one timelike segment by moving slightly the starting
endpoint along the previous segment.

In conclusion if $(p,q)\notin I^{+}$ the continuous causal curve
must be coincident with a lightlike geodesic connecting $p$ to $q$.
This geodesic must be achronal, otherwise there is a timelike curve
$\sigma$ connecting $p',q'\in \gamma$. The continuous causal curve
connecting $p$ to $p'$ following $\gamma$ and $p'$ to $q'$ following
$\sigma$ and $q'$ to $q$ following $\gamma$ is, by the just proved
result, a lightlike geodesic which is impossible since $\sigma$ is
timelike. The contradiction proves that $\gamma$ is achronal.
\end{proof}

Under the assumption of the previous theorem we have:
\begin{corollary}
If $p\ll r$ and $r \le q$ then $p \ll q$. If $p\le r$ and $r\ll q$
then $p\ll q$.
\end{corollary}

\begin{proof}
It follows from the fact that the composition of a timelike and a
causal curve, in whatever order, gives a causal curve which is not a
lightlike geodesic as at some points it is timelike.
\end{proof}

The question as to whether causality theory could be mindlessly
generalized to $C^{1,1}$ spacetime metrics was considered by
Chru\'sciel and Grant \cite{chrusciel12} among others. They
developed a generalization of causality theory to continuous metrics
and found that classical results involving the existence of normal
neighborhoods cannot be proved in their framework. They observed:
\begin{quote}
We note that several statements in [causality theory] concerning
geodesics remain true for $C^{1,1}$ metrics; it is conceivable that
all of them remain true, but justifications would be needed.
\end{quote}
We  provided arguments which prove the correctness of this
expectation. Indeed, the existence of convex normal neighborhoods is
the central technical tool which allows one to complete many local
arguments, such as that on the openness of the chronology relation,
over which causality theory is based. Once the existence of convex
normal neighborhoods has been established, and the local maximization
property of causal geodesics has been obtained, most (if not all)
results of causality theory follow without any substantial
alteration to their classical proofs. Of course those results that
can be expressed only through the use of the second derivative of
the metric, for instance because they use the curvature tensor,
would require further discussion (especially in the Finsler case).

Working with continuous metrics as  in \cite{chrusciel12} expands
very much causality theory though there is a price to be paid. Some
desirable results do not hold anymore, for instance lightlike
geodesics are not necessarily locally achronal (for Lipschitz
connections this result is guaranteed by Theorem \ref{pwa}). For
other differences the reader is referred to \cite{chrusciel12}.

\subsection{Distance balls are convex} \label{nof}

If the spray is the Levi-Civita connection of a Riemannian metric
$g$ it is natural to ask whether the convex neighborhood can be
chosen to be a distance ball, that is, if for sufficiently small
$\delta$, $D_p^{-1}([0,\delta))$ is convex normal. Whitehead gave a
positive answer to this problem through a proof which demands quite
strong differentiability properties on $D_p$ and hence on the
metric. He reasons that it is sufficient to introduce normal
coordinates $(y^1,\cdots, y^n)$ on $N$, because in this way the
distance balls become coincident with the coordinate balls (we
proved the existence of  convex normal neighborhood constructing
them as coordinate balls). Unfortunately, we can apply our argument
which shows the convexity of coordinate balls only for charts which
are $C^{2,1}$, thus the exponential map would have to be $C^{2,1}$
and hence the metric would have to be $C^{3,1}$ (and the manifold
$C^{4,1}$). A different approach \cite{sakai96,wang12} demands just
the twice continuous differentiability of $D_p^2$, but under our
assumptions $D_p^2$ turns out to be just $C^{1,1}$. Finally, one
could invoke Morse's Lemma so as to find a coordinate system in
which the level sets of $D_p^2$ are coordinate spheres
\cite{milnor69}. Unfortunately, this lemma applies only if $D_p^2$
is $C^2$.

Nevertheless, we shall prove that $D_p^2$ is strongly convex using a
Picard-Lindel\"of analysis.

Let us recall \cite{fan03} that a real function defined on an open
convex set $C$ of an affine space $A$ is strongly convex with
constant $\lambda$ if there is a $\lambda>0$ such that for every
$x,y \in C$, $\alpha \in [0,1]$,
\[
f(\alpha x+(1-\alpha) y)\le \alpha f(x)+(1-\alpha)
f(y)-\frac{\lambda}{2}\, \alpha (1-\alpha) \Vert x-y\Vert^2.
\]
A $C^1$ real function is strongly convex on $C$ with constant
$\lambda>0$ if and only if its differential $\dd f$ is strongly
monotone with constant $\lambda$, that is, for every $x,y \in C$,
\[
[\dd f (y)-\dd f(x)]\cdot(y-x)\ge \lambda \Vert y-x\Vert^2.
\]
In a Riemannian space we have analogous definitions and results
\cite[Prop.\ 16.2]{villani09}. A real function defined on an open
geodesically convex set $C$ is geodesically strongly convex with
constant $\lambda$ if there is $\lambda>0$ such that for every
geodesic $x: J\to C$, $\alpha \in [0,1]$,
\[
f( x((1-\alpha)a+\alpha b))\le (1-\alpha) f(x(a))+\alpha
f(x(b))-\frac{\lambda}{2} \,\alpha (1-\alpha) D(x(a),x(b))^2.
\]
A $C^1$ real function is geodesically strongly convex on $C$ with
constant $\lambda>0$ if and only if its differential $\dd f$ is
geodesically strongly monotone with constant $\lambda$, that is, for
every arc-length parametrized geodesic $x: J\to C$,
\[
\dd f(\dot{x}(b))- \dd f(\dot{x}(a)) \ge \lambda D(x(a),x(b)).
\]
The proof of the equivalence is obtained applying the result for the
 affine space case to the  composition $f(x(t))$.

We shall prove

\begin{theorem} \label{keg}
Let  $M$ be a $C^{2,1}$-manifold endowed with a $C^{1,1}$ Riemannian
metric $g$ and corresponding locally Lipschitz Levi-Civita
connection. Let $p\in M$ and let $\epsilon>0$. Let $x^\mu:U\to
\mathbb{R}^n$ be a chart in a neighborhood of $p$ such that,
$g_{\alpha \beta}(p)=\delta_{\alpha \beta}$, $\Gamma^\mu_{\alpha
\beta}(p)=0$. Let $C\subset U$ be a coordinate ball around $p$ which
we already know to be convex normal for sufficiently small radius.
We also have that for sufficiently small radius  for every $q,
q_1,q_2\in C$, interpreting the minus sign and scalar product
through the Euclidean structure induced by the coordinate system on
$C$
\begin{equation} \label{coa}
 \vert [d D^2_q(q_2)- d D^2_q(q_1)] (q_2-q_1)-2(q_2-q_1)^2
\vert  \le \epsilon (q_2-q_1)^2,
\end{equation}
and for every $q\in C$ and arc-length parametrized geodesic $x:J\to
C$
\begin{equation} \label{cob}
 \vert \nabla_{\dot{x}(b)} D^2_q(x(b)) - \nabla_{\dot{x}(a)}
D^2_q(x(a)) -2 D(x(a),x(b)) \vert
 \le \epsilon D(x(a),x(b)).
\end{equation}
In particular, $D^2_q:C\to [0,+\infty)$ is strongly convex with
parameter $\lambda=2-\epsilon$ with respect to {\em both} the
Euclidean
\[
D(q, (1-\alpha) q_1+\alpha q_2)^2\le (1-\alpha) D(q,q_1)^2+\alpha
D(q,q_2)^2-(1-\frac{\epsilon}{2}) \,\alpha (1-\alpha) \Vert
q_2-q_1\Vert^2,
\]
and the Riemannian structures of $C$
\begin{align*}
D(q, x((1-\alpha) a+\alpha b))^2\le (1-\alpha) D(q,x(a))^2&+\alpha
D(q,x(b))^2\\&-(1-\frac{\epsilon}{2}) \,\alpha (1-\alpha)
D(x(0),x(1))^2.
\end{align*}
Thus for any sufficiently small $r$ the open balls $D_p^{-1}([0,r))$ are
contained in $C$ and are (strictly) convex normal neighborhoods.
\end{theorem}

As the balls $D_p^{-1}([0,r))$ are convex we can infer a number of
equivalent convexity properties thanks to the equivalences for
metric spaces proved in \cite{foertsch04}, for instance $D_p$ is
itself convex.

We have the following improvement of our previous formulation of
Gauss' Lemma (Theorem \ref{gux}) which shows that the direct product
sum for the metric can be accomplished on a sphere of a chosen
radius $\bar{r}$ and almost everywhere outside it.
\begin{theorem} \label{oii}
With the notations of the previous theorem, the levels sets
$D_p^{-1}(r)$ are $C^{1,1}$ hypersurfaces diffeomorphic to $S^{n-1}$
with an induced Lipschitz metric $h_r$. Locally on $C$ we can always
find  $C^{2,1}$ functions $\theta_i$ $i=1,\cdots, n-1$, such that in
the $C^{1,1}$ chart $(r,\theta_1,\cdots, \theta_{n-1})$ the metric
takes the form
\begin{equation} \label{kiv}
g=d r^2+(h_r)_{ij}( \dd \theta_i-A_i \dd r) ( \dd \theta_j-A_j \dd
r),
\end{equation}
where $h_r=(h_r)_{ij} \dd \theta_i \dd \theta_j$ and  all the
components $(h_r)_{ij}$, $A_i$, are Lipschitz in $(r,\theta)$.

For any chosen sufficiently small radius $\bar{r}$ new  Lipschitz
coordinates $(r,\alpha_1$, $\cdots,\alpha_{n-1})$ can be found such
that $\{\alpha_i\}$ provide a $C^{1,1}$ chart on $D_p^{-1}(\bar{r})$
and  the metric takes the direct sum form
\[
g=\dd r^2+(h_r(\alpha))_{ij}\,\dd \alpha^i \dd \alpha^j,
\]
where the components $(h_r(\alpha))_{ij}$ are defined almost
everywhere and are bounded (this holds also for spherical normal
coordinates) and for $r=\bar{r}$ they are defined everywhere and are
Lipschitz in $\alpha$.
\end{theorem}
%
%
%
%

\subsection{Convexity on Lorentzian manifolds}
In this section we study the problem of the existence of
convex/concave functions and sets in Lorentzian manifolds. The
Lorentzian case is more involved than the Riemannian but is
essential for the understanding of Lorentzian manifolds under low
differentiability conditions.

 In the next theorem $\eta_{\alpha \beta}$ is the usual
Minkowski metric in diagonal form ($\eta_{00}=-1$, $\eta_{ii}=+1$
for $i\ge 1$).
\begin{theorem} \label{igy}
Let  $M$ be a $C^{2,1}$-manifold endowed with a $C^{1,1}$ Lorentzian
metric $g$ and corresponding Lipschitz Levi-Civita  connection. Let
$p\in M$ and let $\epsilon>0$. Let $x^\mu:U\to \mathbb{R}^n$ be a
chart in a neighborhood of $p$ such that, $g_{\alpha
\beta}(p)=\eta_{\alpha \beta}$, $\Gamma^\mu_{\alpha \beta}(p)=0$.
Let $C\subset U$ be a coordinate ball around $p$ which we already
know to be convex normal for sufficiently small radius. We also have
that for sufficiently small radius, for every geodesic $x: [0,1]\to
C$,  and for every $q\in C$
\begin{equation} \label{pat}
\vert \frac{\dd }{\dd t}D^2_q(x(t))\vert_{t=1}-\frac{\dd }{\dd t}
D^2_q(x(t))\vert_{t=0}-2D^2(x(0),x(1)) \vert\le \epsilon \Vert
x(1)-x(0)\Vert^2,
\end{equation}
where the Euclidean and affine structure of the coordinate chart is
used just on the right-hand side.
\end{theorem}

Observe that in the next theorem there is no mention to the Euclidean
or affine structures induced by a coordinate chart (we stress once again that $D^2_p$ can be negative).

\begin{theorem} \label{jse}
Let  $M$ be a $C^{2,1}$-manifold endowed with a $C^{1,1}$ Lorentzian
metric $g$ and corresponding locally Lipschitz Levi-Civita
connection. Let $p\in M$ and let $\epsilon>0$.  Let $\gamma: I\to
M$, $t\mapsto \gamma(t)$, be a  timelike geodesic such that
$p=\gamma(0)$. The convex normal set $C\ni p$ of Theorem \ref{nsx} can be chosen so small
that once $I$ is redefined to be the connected component of
$\gamma^{-1}(C)$ containing $0$, the following property holds.

For every $q,r\in \gamma(I)$, $q\ne r$, there is a strictly convex
normal  set $O\ni r$, $\bar{O}\subset C$, such that  all the points
of $\bar{O}$ are either in the chronological future or in the
chronological past of $q$; every geodesic $x: J\to O$, connecting
two points on the same level set $(D_q^2)^{-1}(c)$, $c<0$, of the
function $D^2_q\colon C\times C \to \mathbb{R}$,  satisfies, once
reparametrized with respect to $g$-arc length ($x$ is necessarily
spacelike by Theorem \ref{pwa}, thus $D(x(a),x(b))\ge 0$), for every $a,b\in J$,
\begin{equation} \label{lor}
  \vert \nabla_{\dot{x}(a)} D^2_q(x(a)) - \nabla_{\dot{x}(b)}
D^2_q(x(b))-2 D(x(a),x(b))\vert \le
 \epsilon D(x(a),x(b)).
\end{equation}
In particular, $D^2_q(x(t))$ is strongly convex with parameter
$2-\epsilon$
\[
D_q^2(x((1-\alpha) a +\alpha b))\le (1-\alpha) D_q^2(x(a))+\alpha
D_q^2(x(b))-(1-\frac{\epsilon}{2}) \,\alpha (1-\alpha)
D(x(a),x(b))^2,
\]
and the sets of the form $(D^2_q)^{-1}((-\infty,c))\cap O$ for $c<0$
are strictly geodesically convex.
\end{theorem}

\begin{remark}
Mimicking the proof of \cite[Prop.\ 3.1]{andersson98} it is easy to
show that the Lorentzian distance on $O$ from $q$,
$D^L_q=\sqrt{-D^2_q}$ is semiconvex and hence almost everywhere
first and twice differentiable. However, in order to prove the
geodesic convexity of $(D^2_q)^{-1}((-\infty,c))\cap O$ we need a
result on the convexity of $D^2_q$ rather than on the convexity of
$D^L_q$. Furthermore, we stress that the notion of semiconvexity, due to Rockafellar, is a rather weak notion. For
instance the concave function $-x^2$ on $\mathbb{R}$ is semiconvex
(thus the terminology {\em semiconvexity} can be misleading).
\end{remark}

We recall that a subset $A \subset M$ is {\em causally convex} in $B\subset M$,
with $A\subset B$, if every $C^1$ causal curve contained in $B$ and
joining two points of $A$ is necessarily contained in $A$. An open subset $B$ is {\em strongly causal}
if every point $p\in B$ admits arbitrarily small open neighborhoods  which are causally convex in $B$.
An open subset $S$ of $M$ is called {\em causally
simple} if it is strongly causal and $J^+_S\subset S\times S$ is closed in the product topology \cite{hawking73,minguzzi06c}. A causally simple
subset $S$ is {\em globally hyperbolic} if for every $p,q\in S$,
$J_S^{+}(p)\cap J^{-}_S(p)$ is compact.

The following claim was  known under stronger differentiability assumptions. It can be found in a footnote of
\cite{minguzzi06c}.
\begin{theorem} \label{bik}
Convex normal subsets  are causally simple.
\end{theorem}

\begin{proof}
Let $T$ be the global timelike vector field which provides the time
orientation. Let $C$ be a convex normal subset, and let
$f_1,f_2\colon C\times C \to \mathbb{R}$ be the functions
$f_1(p,q):=g(\exp_p^{-1}q, \exp_p^{-1}q)$,
$f_2(p,q):=g(\exp_p^{-1}q, T(p))$ since $\exp^{-1}$ and $g$ are
continuous $f_1$ and $f_2$ are continuous, and hence
$J^{+}_C=f_1^{-1}((-\infty,0])\cap f_2^{-1}((-\infty,0))$ is closed.

A spacetime $C$ is strongly causal if and only if for every $p,q \in C$, $(p,q)\in J_C^+$ and $(q,p)\in \overline{J_C^+}$ imply $p=q$ (see \cite{minguzzi07b}).
Suppose that there are $p,q\in C$, $p\ne q$, such that $p\le_C q$
and $q\le_C p$ (we just proved $\overline{J_C^+}=J^+_C$). Let $\gamma_1$ be the future directed causal geodesic
connecting $p$ to $q$ and let $\gamma_2$ be the future directed
causal geodesic connecting $q$ to $p$. Then the images of $\gamma_1$
and $\gamma_2$ differ (otherwise there would be a geodesic which is
both future and past directed), and hence there are two
geodesics connecting $p$ to $q$, a contradiction to the uniqueness
of the connecting geodesic in convex normal sets (Theor.\ \ref{nsx}).
\end{proof}

\begin{corollary}\label{cop}
Let  $M$ be a $C^{2,1}$-manifold endowed with a $C^{1,1}$ Lorentzian
metric $g$ and corresponding locally Lipschitz Levi-Civita
connection. Let $p\in M$  then there is a local base $\{C_i\}$ for
the topology at $p$ such that for every $i$, $C_i$ is strictly
convex normal, globally hyperbolic, $\bar{C}_{i+1}\subset C_i$, and
$C_{i+1}$ is causally convex in $C_{i}$ (and hence $C_1$).
\end{corollary}

Part of the previous result was already known
\cite{minguzzi06c}. What was open was the result on the convexity of
$C_i$. Observe that if $C_i$ is globally hyperbolic then any two
causally related events are connected by a causal geodesic contained
in $C_i$ (by the Avez-Seifert theorem \cite{hawking73}). Thus what
was really missing was the convexity with respect to spacelike
geodesics. These globally hyperbolic convex normal sets behave rather well, indeed this property is left invariant under finite intersections.  As a consequence, by Lebesgue's covering lemma, for any compact  set $K$ on spacetime, for any metric $\rho: M\times M \to \mathbb{R}$ inducing the manifold topology, and finite covering of $K$ by globally hyperbolic convex normal sets, there is an $\epsilon>0$  such that for any pair of points $p,q$ such that   $\rho(p,q)<\epsilon$, points $p$ and $q$ are contained in one element of the covering and hence connected by some geodesic.

\begin{remark}
A spacetime is strongly causal if every point admits an arbitrarily
small causally convex set (in $M$). In a strongly causal spacetime,
we can find a causally convex (in $M$) open neighborhood $Y$ of $p$
contained in $C_1$, and for sufficiently large $i$, $C_i\subset Y$.
As a consequence, the sets $C_i$ for sufficiently large $i$ are also
causally convex in $M$. Thus for strongly causal spacetimes we can
include in the previous Corollary the causal convexity of $C_i$ with
respect to $M$ among the properties of these sets.
\end{remark}

%
%
%


\subsection{Two variations on the main theme}
Our results admit a number of variations obtained  replacing the
base point $p$ of the pointed exponential map with an embedded
manifold. For instance, consider a pseudo-Riemannian manifold
endowed with a $C^{1,1}$ metric and a Lipschitz  metric compatible
connection.
 Let $\phi:S\to M$, $k\ge 1$, be a
$C^{1,1}$ $k$-dimensional embedding, such that the induced metric is
pseudo-Riemannian. Its non-degeneracy implies that at each point
$p\in \phi(S)$ the tangent space $T_pM$ is the direct sum of the
tangent space to $\phi(S)$ and it normal space. Let $\nu(S)$ be the
corresponding $n$-dimensional normal bundle with base $\phi(S)$.


We shall prove the following theorem analogous to
 to Theorem \ref{jfd}:

\begin{theorem} \label{jfe} $\empty$
\begin{itemize}
\item[($\exp_{\nu(S)}$)]
The vector bundle $\nu(S)$ is Lipschitz. Moreover, the map $\exp_{\nu(S)}$ is strongly differentiable on the image of
the zero section of $\pi\vert _{\nu(S)}: \nu(S)\to \phi(S)$, and establishes a
Lipeomorphism between a neighborhood of the image of the zero section and a
neighborhood of $\phi(S)$.
\end{itemize}
\end{theorem}

By the Lipschitzness of $\nu(S)$ we can apply to
$\exp_{\nu(S)}:=\exp\vert_{\nu(S)}$ the results of Theorem
\ref{flo}. This theorem can also be used to construct tubular
neighborhoods. The pointed exponential map can be regarded as a
special case of this type of construction for $k=0$.

We can also consider a function $H^\mu$ dependent on time provided
it is Lipschitz. Let $\gamma_{(t_0,v)}(t)$ be any solution of
\begin{align}
\frac{\dd x^{\mu}}{ \dd t}&=v^\mu,  \\
\frac{\dd v^\mu}{\dd t}&=H^\mu(t,x,v)  ,
\end{align}
with initial condition $v\in TM$ at time $t_0$.
\begin{theorem}
The solution $\gamma_{(t_0,v)}(t)$ is locally Lipschitz in
$(t,t_0,v)$ and $\gamma_{(t_0,v)}(\cdot)$ is $C^{2,1}_{loc}$. Let
$(\hat{t},\hat{p})\in \mathbb{R}\times M$, then for every
$\epsilon>0$ there is an open neighborhood $C$ of $\hat{p}$ such
that for every $p_1,p_2\in C$, $t_1,t_2\in
(t_0-\epsilon,t_0+\epsilon)$, $t_1<t_2$, there is one and only one
solution starting from $p_1$ at time $t_1$ and reaching $p_2$ at
time $t_2$  entirely contained in $C$.
\end{theorem}
The idea is to rewrite the system of first order ODE as follows
\begin{align*}
\frac{\dd t}{\dd s}&=t', \\
\frac{\dd x^{\mu}}{ \dd s}&={x'}^\mu,  \\
\frac{\dd t'}{\dd s}&=0, \\
\frac{\dd {x'}^\mu}{\dd s}&=(t')^2\,H^\mu(t,x,x'/t')=H^\mu(t,x,x') ,
\end{align*}
where $x'= v t'$, so as to reduce the problem to the ($s$-)time
independent case. The previous theorem is then a corollary of
Theorem \ref{nsx} whenever in the proof provided for that theorem in
place of the Banach space $(\mathbb{R}^n, \Vert\, \Vert)$ we
consider the Banach space $(\mathbb{R}^{n+1}, \max(\Vert\, \Vert,
\vert \, \vert))$.

%
%
%
%

\subsection{Some technical preliminary results} \label{jtt}



Let us first recall that in finite dimensions and for Lipschitz
functions the notions of G\^ateaux differential and Frechet
differential coincide \cite[p.\ 158]{dieudonne69}.

\begin{proposition} \label{gat}
Any Lipschitz function $f\colon O\to \mathbb{R}$, $O\subset
\mathbb{R}^n$, which is G\^ateaux differentiable at $p$ is
differentiable at $p$.
\end{proposition}
\begin{proof}
Let $G$ be the G\^ateaux differential at $p$. By contradiction,
suppose that  $G$ is not the differential of  $f$  at $p$, then
there is $\epsilon>0$ and a sequence $v_n\in \mathbb{R}^n$, $v_n \ne
0$, $v_n\to 0$, such that
\begin{equation} \label{ydf}
\Vert f(p+v_n)-f(p)-G(v_n)\Vert> \epsilon \Vert v_n\Vert.
\end{equation}
Let $e_n=v_n/\Vert v_n\Vert$, by the compactness of $S^{n-1}$ we can
assume without loss of generality that $e_n \to e$, with $\Vert
e\Vert=1$. Let us decompose $v_n$ in  components parallel and
perpendicular to $e$
\[
v_n=a_n e+b_n
\]
where $a_n$ is a scalar and $b_n$ is a vector. We have
\begin{align*}
\Vert f(p+v_n)-f(p)-G(v_n)\Vert&\le \Vert f(p+v_n)-f(p+a_n e)\Vert
\\&
\quad +\Vert f(p+a_n e)-f(p)-G(a_n e)\Vert
\\&
\quad +\Vert
G(a_n e)-G(v_n)\Vert\\
&= K \Vert b_n\Vert+o_{e}(a_n)+\vert G\vert \, \Vert b_n\Vert ,
\end{align*}
where $K$ is the Lipschitz constant. The condition $e_n\to e$ reads
$a_n/\Vert v_n\Vert \to 1$, $\Vert b_n\Vert / \Vert v_n\Vert\to 0$.
For sufficiently large $n$ the previous inequality contradicts Eq.\
(\ref{ydf}).

\end{proof}

In order to prove Theorem \ref{flo} we shall need the following
result on differentiation under the integral sign which, we believe,
is interesting in its own right.

\begin{theorem} \label{hua}
Let $f\colon [a,b]\times \mathbb{R}^k\to \mathbb{R}$, $(x,y) \mapsto
f(x,y)$, be a continuous function which is locally Lipschitz in $y$
uniformly in $x$. Then for almost every $y$, the differential $d_2
f$ exists at $(u,{y})$  for almost every $u\in [a,b]$, it is
summable in $[a,b]$, and for every $x\in [a,b]$:
\[
(\dd_2 \int_a^x f(u,y)\, \dd u)=  \int_a^x \dd_2 f(u,{y})\,
\dd u.
\]
\end{theorem}

\begin{proof}
Let $\bar{y}\in \mathbb{R}^n$, there is a relatively
compact neighborhood $O\ni \bar{y}$ which is a product of open
intervals.
The function $f(x,\cdot)\vert_O$ is $K$-Lipschitz in $y$ where $K$ does not depend
on $x$. In particular, for any $x\in [a,b]$ the function
$f(x,\cdot)$ is  differentiable almost everywhere (Rademacher's
theorem). Let $E(x)\subset O$ be the subset where $f(x,\cdot)\vert_O$ is
differentiable.  Fubini's theorem applied to the characteristic function
of
\[
A=\cup_{x\in [a,b]} [\{x\}\times E(x)]\subset [a,b]\times
O,
\]
states that for almost every $y \in O$, the differential $d_2 f$
exists at $(x,{y})$  for almost every $x\in [a,b]$. From now on let
$y$ be one of these special values.

As a consequence, for almost every $y \in O$, and for every
vector $v\in \mathbb{R}^k$, the partial derivative $\p_v f(x,y)$ exists for
almost every $x\in [a,b]$. Because of the Lipschitz condition we
have $\vert \p_v f(x,y)\vert\le K\Vert v\Vert $. Let $\epsilon_n\to 0$ then
\[
\p_v f(x,y)=\lim_{n\to \infty} f_n^y(x) ,
\]
where
\[
f_n^y(x)=\frac{1}{\epsilon_n} [f(x,y+v \epsilon_n)-f(x,y)],\qquad
\vert f^y_n(x) \vert \le K\Vert v\Vert  .
\]
By the dominated convergence theorem $\p_v f(x,y)$ is
summable and
\[
\p_v \int_a^x f(u,y) \dd u= \int_a^x \p_v f(u,y) \dd u.
\]
This equation proves that the linear  operator $G(v)$ on the
right-hand side is the G\^ateaux differential at $y$ of
$F(x,v):=\int_a^x f(u,v) \dd u$. Moreover, this function is
Lipschitz in $v$ thus $F(x,\cdot)$ is differentiable and  $G$
coincides with its differential (Prop.\ \ref{gat}).
\end{proof}

%
%

Alternatively, we could use a theorem \cite{minguzzi13b} which
improves Schwarz's theorem on the equality of mixed partial
derivatives. The reader is referred to \cite{minguzzi13b} for
further details. A result similar to the next one was proved by  Federer \cite[Lemma 4.7]{federer59}.

%
%
%

\begin{theorem} \label{pok}
Let $f\colon O\to \mathbb{R}$ be a Lipschitz function defined on an
open subset of $\mathbb{R}^n$. Let $J\colon O\to \mathbb{R}^n$ be a
continuous function. If the differential of $f$ exists and coincides
with $J$ almost everywhere then $f$ is $C^{1}$ and its differential
is $J$.
\end{theorem}

\begin{proof}
For $n=1$ this statement follows from inspection of the function
\[
f(x)-f(a)-\int^x_a J(u) \dd u,
\]
$a\in O$. This function is
Lipschitz, thus absolutely continuous, and hence has a derivative which vanishes almost everywhere,
thus it is a constant. Taking $x=a$ we  find that this constant is
zero, thus $f(x)=f(a)+\int^x_a J(u) \dd u$ which implies that $f$ is
$C^1$ with derivative $J$.

We can assume without loss of generality,  $O=\mathbb{R}^n$.
Let $v\in \mathbb{R}^n$, we want to show that at every $x\in O$,
$\p_v f=J(v)$. This fact would imply that $f$ is G\^ateaux
differentiable at each point with G\^ateaux differential $J$, thus
the desired conclusion would follow from Prop.\ \ref{gat}.

By Fubini's and Rademacher's theorem for almost every hyperplane perpendicular to
$v$, the function $f$ is almost everywhere differentiable at the
points belonging to the hyperplane.


Let us introduce coordinates $\{y^0,\cdots, y^n\}$ such that one
such hyperplane has equation $y^0=0$ and $v=\p_{y^0}$. Let
$g(y_1,\cdot, y_n):=f(0,y_1,\cdots,y_n)$ be the Lipschitz
restriction of $f$ to the hyperplane.

The function $J(t,y)$ is continuous and hence uniformly continuous
over compact subsets. As a consequence,  the
function
\[
F(t,y)=g(y)+\int_0^t J(t,y;e_0) \, \dd t
\]
is continuous in $(t,y)$ and $C^{1}$ in $t$.

%


Let $(t,y)\in \mathbb{R}^n$ and let $U\ni (t,y)$ be an open
neighborhood.  The function $f$ is almost everywhere differentiable
with differential $J$. By Fubini's theorem for almost every line
parallel to the $y^0$-axis ( the line is determined by its
intersection with $y^0=0$ and  ``a.e.\ $\!\!\!$'' in this statement is meant
in the Lebesgue ($n-1$)-dimensional measure of this hyperplane), $f$
is almost everywhere differentiable with differential $J$. Thus
there is some $(t',y')\in U$ passing through one such line. The
function $f-F$ over such a line is Lipschitz and differentiable almost
everywhere with zero derivative, thus it is a constant. But $f-F=0$
at $y^0=0$ thus the constant vanishes and hence
$f(t',y')=F(t',y')$. As both functions are continuous and $U$ is
arbitrary, $f(t,y)=F(t,y)$ that is $f=F$. We conclude that $\p_{v}f=
\p_v F=\p_{y^0} F= J(e_0)$, which is what we wanted to prove.
\end{proof}

\section{Proofs I: A Picard-Lindel\"of analysis}

This section is devoted to the proofs of the results stated in the
previous section and, in particular, to the proof of Theorem
\ref{nsx}. In order to prove the strong differentiability of $\exp$
we shall make a Picard-Lindel\"of analysis of the geodesic equation
in the small. We shall also pass through the proof of existence and Lipschitz dependence on initial conditions. Though
they are already known they  will be useful to fix  the notation and
introduce the bounds used in the last step of the proof.

Let $x^\mu\colon U\to \mathbb{R}^n$ be a coordinate chart in a neighborhood $U$
of $p\in M$. Without loss of generality let us assume that
$x^\mu(p)=0$ and let $r$ be such that the closed ball
$\bar{B}(p,r)=\{q: \Vert q-p\Vert\le r\}$ is contained in $U$, where
$\Vert \, \Vert$ is the Euclidean coordinate norm. On the tangent
bundle we introduce the local coordinate system $\{x^\mu,\dot
x^\mu\}$.  We shall regard the coordinate chart image $x^\mu(U)$  as an open subset of
the normed space $(\mathbb{R}^n, \Vert \,\Vert)$.

Let us consider the system (\ref{fhs})-(\ref{fha}) where $H^\mu$ is
homogeneous of second degree in $v$ and Lipschitz
\[
\Vert H(x_2,v_2)-H(x_1,v_1)\Vert\le \alpha \Vert x_2-x_1\Vert+\beta
\Vert v_2-v_1\Vert
\]
in the domain $\bar{B}(p,r)\times \{v: \Vert v\Vert\le 1\}$.

Suppose that $v_1,v_2$ are not bounded as stated. Let $V$ be any
constant such that
\begin{equation} \label{uhu}
V>\max(\Vert v_1\Vert, \Vert v_2\Vert).
\end{equation}
The Lipschitz conditions is then better rewritten in the following
form:
\begin{align*}
\Vert H(x_2,v_2)-H(x_1,v_1)\Vert&=V^2 \Vert
H(x_2,\frac{v_2}{V})-H(x_1,\frac{v_1}{V})\Vert\\
&\le V^2\{\alpha \Vert x_2-x_1\Vert+\beta \Vert
\frac{v_2}{V}-\frac{v_1}{V}\Vert \}\\
&\le \alpha V^2\Vert x_2-x_1\Vert +\beta V\Vert v_2-v_1\Vert.
\end{align*}
Let
\begin{equation} \label{emm}
M=\sup_{x\in \bar{B}(p,r)}  \sup_{ \Vert v\Vert=1} \Vert H(x,
v)\Vert.
\end{equation}
%

We rewrite  (\ref{fhs})-(\ref{fha}) in integral form
\begin{align}
x^\mu(t, x_0,\dot{x}_0)&=x^\mu_0+\int_0^t \dot{x}^\mu(t,x_0,\dot{x}_0)\, \dd t , \label{g1}\\
\dot{x}^\mu(t, x_0,\dot{x}_0)&= \dot{x}^\mu_0+\int_0^t H^\mu(x(t,
x_0,\dot{x}_0), \dot{x}(t, x_0,\dot{x}_0)) \, \dd t , \label{g2}
\end{align}
where $(x_0,\dot{x}_0)$ is the initial condition at time $t=0$. We
have included in the above expression the dependence of
$(x,\dot{x})$ on the initial conditions. Unless  needed we shall
remove it from the expressions below.

Let $(x_0,\dot{x}_0)$
belong to the domain
\begin{align}
\max\{ \Vert x_0\Vert,\Vert \dot x_0\Vert\} &< \delta, \label{d1}
\end{align}
where  $\delta$ is a positive constant such that (the expression
makes sense for $M=0$)
\begin{equation} \label{jux}
  \delta < \frac{1}{M}(1-e^{-M r/2})\le \frac{r}{2} ,
\end{equation}
and is sufficiently small that
\begin{equation} \label{juz}
\frac{\delta}{1-\delta M}\le 1,\qquad \frac{\beta \delta}{2(1-\delta\, M)} \, \big(1+\sqrt{1+4 \alpha/\beta^2}\, \big)\le 1.
\end{equation}

Starting from $k=-1$ with
\[
x_{-1}(t)=x_0, \quad \dot{x}_{-1}(t)=0,
\]
we define inductively the next two functions defined over $[0,1]$ which, in the induction
hypothesis are  well defined  and $C^1$  and are such that  $(x_k,\dot x_k)(t)\in \pi^{-1}(\bar{B}(p,r))$ for every $t\in [0,1]$
\begin{align}
x^\mu_{k+1}(t)&=x^\mu_0+\int_0^t \dot{x}_k^\mu(t)\, \dd t , \label{p1}\\
\dot{x}^\mu_{k+1}(t)&= \dot x^\mu_0+\int_0^t
H^\mu(x_k(t),\dot{x}_k(t)) \, \dd t . \label{p2}
\end{align}
In particular observe that they imply for $k=0$
\[
x_{0}(t)=x_0, \quad \dot{x}_{0}(t)=\dot{x}_0,
\]
while for $k\ge 1$ they will generically depend on $t$. From Eqs.\ (\ref{p1})-(\ref{p2}) we obtain  (we shall repeatedly use the inequality $\Vert \int v^\mu(t)
\dd t\Vert\le \int \Vert v^\mu(t)\Vert \dd t$ which is well known
once interpreted as the fact that the length of a $C^1$ curve on
$\mathbb{R}^n$  is greater than the distance between its endpoints)
\begin{align*}
\Vert x_{k+1}(t)\Vert &\le \Vert x_0\Vert +\int_0^t \Vert \dot{x}_k(t)\Vert\, \dd t ,\\
\Vert \dot{x}_{k+1}(t)\Vert &\le  \Vert \dot x_0\Vert +\int_0^t M
\,\Vert \dot{x}_k(t)\Vert^2\, \dd t .
\end{align*}
The second inequality implies inductively the following bound
\[
\Vert \dot{x}_{k+1}(t)\Vert \le \frac{\Vert \dot x_0\Vert}{1-\Vert
\dot x_0\Vert M t} ,
\]
which is clearly satisfied for $k=-1$ and which, replaced into the first inequality gives
\[
\Vert x_{k+1}(t)\Vert\le \Vert x_0\Vert -\frac{1}{M}\ln (1-\Vert
\dot x_0\Vert M t )\le \delta- \frac{1}{M}\ln (1-\delta \, M ).
\]
By (\ref{jux}) we have for every $k\ge 0$,
\begin{align*}
 \Vert x_{k}(t)\Vert& < r,  \\
\Vert \dot x_{k}(t)\Vert& < \frac{\delta}{1-\delta\, M}
\end{align*}
thus these functions define indeed points belonging to  $\pi^{-1}(\bar{B}(p,r))$.
We observe that for any instants
$t_1,t_2\in [0,1]$ and any pair $v_1:=\dot{x}_i(t_1), v_2:=\dot{x}_j(t_2)$
\begin{equation} \label{pki}
V:=\frac{\delta}{1-\delta\, M}
\end{equation}
satisfies the condition of Eq.\ (\ref{uhu}). From now on $V$ will be
given by this equation.

\subsection{Existence of geodesics} \label{oyt}

We have:
\begin{align*}
\Vert x_{k+1}(t)- x_{k}(t)\Vert &\le \int_0^t \Vert \dot{x}_k(t)-\dot{x}_{k-1}(t) \Vert \, \dd t ,\\
\Vert \dot{x}_{k+1}(t)-\dot{x}_{k}(t) \Vert & \! \le \!\! \int_0^t
\!\! \Vert H(x_k(t),\dot{x}_k(t)) -H(x_{k-1}(t),\dot{x}_{k-1}(t))
\Vert
\, \dd t , \\
&\le \int_0^t\{A \Vert x_k(t)-x_{k-1}(t)\Vert + B\Vert
\dot{x}_{k}(t)-\dot{x}_{k-1}(t)\Vert \}\, \dd t ,
\end{align*}
where $A=[\frac{\delta}{1-\delta\, M}]^2 \alpha$ and $B= \frac{\beta
\delta }{1-\delta\, M}$. There is a positive constant $D$ such that
\begin{equation} \label{kiu}
\frac{A}{D}+B=D,
\end{equation}
 namely
\begin{equation} \label{fun}
D=\frac{1}{2}\,\big( B+\sqrt{B^2+4 A}\,\big)=\frac{\beta \delta}{2(1-\delta\,
M)} \, \big( 1+\sqrt{1+4 \alpha/\beta^2}\, \big).
\end{equation}
By Eq.\ (\ref{juz})
\[
D\le 1.
\]
By induction we have the bounds
\begin{align*}
\Vert {x}_k(t)-{x}_{k-1}(t) \Vert &\le  \frac{1}{D}\frac{(Dt)^k}{k!} , \\
\Vert \dot{x}_k(t)-\dot{x}_{k-1}(t) \Vert &\le \frac{(Dt)^k}{k!} .
\end{align*}
They imply that the series
\begin{align*}
x^\mu_k(t)&= x^\mu_0+ \sum_{k=0}^n [x^\mu_{k+1}(t)-x^\mu_k(t)], \\
\dot{x}^\mu_k(t)&=\dot x^\mu_0+\sum_{k=0}^n [\dot
x^\mu_{k+1}(t)-\dot x^\mu_k(t)],
\end{align*}
define a succession of continuous functions which converge uniformly
to (continuous)  functions $x^\mu(t)$ and $\dot x^\mu(t)$ over
$[0,1]$. By uniform convergence we can pass to the limit in Eqs.\
(\ref{p1})-(\ref{p2}). Indeed, observe that $H^\mu(x,\dot{x})$ is
continuous over the compact set $\bar{B}(p,r)\times \{v: \Vert
v\Vert \le V\}$ and hence uniformly continuous. Thus these limits
indeed solve the  system (\ref{g1})-(\ref{g2}). In particular, Eq.\
(\ref{g1}) proves that indeed $\dot x^\mu(t)=\frac{\dd }{\dd t}\,
x^\mu(t)$.

For the proof of the uniqueness of geodesics the reader can consult \cite[Theor.\ 1.1, Chap.\ 2]{hartman64}.

%
%
%
%
%
%
%

\subsection{Lipschitz dependence on initial conditions {\normalsize(Theorem \ref{flo})}} \label{dep}

The Lipschitz dependence on the initial condition can be proved
using the Gronwall's inequality.  It is rather easy to obtain it
directly using the Picard-Lindel\"of approximation.

Let us consider two solutions $x^\mu(t)$ and $y^\mu(t)$ of the
geodesic equation with initial conditions $(x^\mu_0,\dot x^\mu_0)$,
$(y^\mu_0,\dot y^\mu_0)$,  which belong to the domain given by Eq.\
(\ref{d1}). Starting from these initial conditions we define
inductively functions $x_k^\mu(t)$, $y_k^\mu(t)$ as above, and
subtract the corresponding Eqs.\ (\ref{p1})-(\ref{p2}).
\begin{align*}
\Vert x_{k+1}(t)- y_{k+1}(t)\Vert &\le \Vert x_0-y_0\Vert+\int_0^t
\Vert
\dot{x}_k(t)-\dot{y}_{k}(t) \Vert \, \dd t ,\\
\Vert \dot{x}_{k+1}(t)-\dot{y}_{k+1}(t) \Vert &\le \Vert \dot
x_0-\dot y_0\Vert+\!\!\int_0^t\!\{A \Vert x_k(t)-y_{k}(t)\Vert +
B\Vert \dot{x}_{k}(t)-\dot{y}_{k}(t)\Vert \}\, \dd t .
\end{align*}
We now regard the chart-trivialized tangent bundle of coordinates
$\{x^\mu, \dot x^\mu\}$ as the direct sum of vector spaces
$\mathbb{R}^n\oplus \mathbb{R}^n$ endowed with the norm
\[
\Vert (x, \dot x) \Vert:= \max \{\Vert x \Vert, \Vert \dot x \Vert
\}.
\]
By induction we obtain
\begin{align*}
\Vert x_{k}(t)- y_{k}(t)\Vert &\le  \max \{ D \Vert
x_0-y_0\Vert,\Vert \dot x_0-\dot y_0\Vert\}\,
\frac{e^{Dt}}{D}  , \\
\Vert \dot{x}_{k}(t)-\dot{y}_{k}(t) \Vert & \le \max \{D\Vert
x_0-y_0\Vert,\Vert \dot x_0-\dot y_0\Vert\} \, e^{Dt}.
\end{align*}
Clearly, the induction hypothesis is  satisfied for $k=0$. Taking the limit $k\to \infty$
\begin{align*}
\Vert x(t)- y(t)\Vert &\le  \max \{ D \Vert x_0-y_0\Vert, \Vert \dot
x_0-\dot y_0\Vert\}\,
\frac{e^{Dt}}{D}  , \\
\Vert \dot{x}(t)-\dot{y}(t) \Vert & \le \max \{D\Vert x_0-y_0\Vert,
\Vert \dot x_0-\dot y_0\Vert\} \, e^{Dt}.
\end{align*}
Recalling that $D \le 1$ these inequalities imply
\[
 \Vert (  x, \dot{x})(t) - (y,\dot{y})(t)\Vert  \le \Vert
(x_0,\dot x_0) -(y_0, \dot y_0)\Vert \, \frac{e^{Dt} }{D} ,
\]
which proves the Lipschitz dependence on the initial
conditions (the exponential map is obtained for $t=1$).

The joint dependence on $t$ and $(x_0,\dot{x}_0)$ is also locally
Lipschitz, indeed since the time dependence of $( x, \dot{x})(t)$ is
$C^{1,1}$ it is locally Lipschitz, thus for $t,t'\in [a,b]$
\[
\Vert (  x, \dot{x})(t') - (y,\dot{y})(t)\Vert \le  \Vert (  x,
\dot{x})(t')- (  x, \dot{x})(t)\Vert + \Vert (  x, \dot{x})(t) -
(y,\dot{y})(t)\Vert ,
\]
from which we infer the local Lipschitzness  of the dependence on
$(t,x_0,\dot{x}_0)$.

As a consequence, the geodesic flow over $TM$ is locally Lipschitz
because  every geodesic segment can be covered with a finite number
of coordinate patches. By continuity the domain $W$ where it is
defined is open and satisfies the conditions of Theorem \ref{flo}.
Analogously, $\Omega$ is open.

 Let us rewrite the system
(\ref{g1})-(\ref{g2}) reintroducing the dependence on the initial
conditions
\begin{equation} \label{bhq}
(x,\dot x)(t,x_0,\dot x_0)=(x_0,\dot x_0)+\int_0^t (f_1, f_2)
(t,x_0,\dot x_0) \dd t ,
\end{equation}
where
\begin{align*}
f_1^\mu(t,x_0,\dot x_0)&=\dot x^\mu(t,x_0,\dot x_0), \\
f_2^\mu(t,x_0,\dot x_0)&=H^\mu(x(t,x_0,\dot x_0), \dot{x}(t,x_0,\dot
x_0)).
\end{align*}
Since the dependence $(x,\dot x)(t,x_0,\dot x_0)$ is Lipschitz we
can apply Theorem \ref{hua} with the replacements $x\to t$, $y\to
(x_0,\dot x_0)$.

We conclude that for almost every $(x_0,\dot{x}_0)$
the differential $\dd_{(x_0,\dot{x}_0)} (f_1, f_2) (u)$ exists for almost every $u$, it is summable and for every $t\in [0,1]$
\begin{equation} \label{dow}
\dd_{(x_0,\dot{x}_0)}(x,\dot x)=\dd_{(x_0,\dot{x}_0)}(x_0,\dot
x_0)+\int_0^t \dd_{(x_0,\dot{x}_0)} (f_1, f_2)(u,x_0,\dot x_0) \dd
u.
\end{equation}
In particular, for almost every $(x_0,\dot{x}_0)$ satisfying Eq.\
(\ref{jux}) the differential $\dd_{(x_0,\dot{x}_0)}(x,\dot x)$
exists for every $t\in [0,1]$ and is given by this equation. Let us
call $U$ this special subset of the initial conditions.


Since the functions $f_1$ and $f_2$ are locally Lipschitz the
differential on the right-hand side is bounded and so, for every
$(x_0,\dot{x}_0)\in U$, the quantity $\dd_{(x_0,\dot{x}_0)}(x,\dot
x)$ has a Lipschitz dependence on $t$ where the Lipschitz constant
does not vary in a small relatively compact neighborhood of
$(x_0,\dot{x}_0)$. In other words, $\dd_{(x_0,\dot{x}_0)}(x,\dot x)$
is a function dependent on $(t,x_0,\dot{x}_0)$ which is Lipschitz in
$t$ uniformly in those $(x_0,\dot{x}_0)$ belonging to $U$.

Differentiating Eq.\ (\ref{dow}) with  $(x_0,\dot{x}_0)\in U$ we get
that for almost every $t$
\[
\frac{\dd }{\dd t}\, \dd_{(x_0,\dot{x}_0)}(x,\dot
x)=\dd_{(x_0,\dot{x}_0)} (f_1, f_2)= \dd_{(x_0,\dot{x}_0)} \frac{\dd
}{\dd t}\, (x,\dot x).
\]
Though we performed just a local analysis, the conclusion does not
change in the setting of Theorem \ref{flo} where $U$ is replaced by
$\tilde{\Omega}$ since, as observed above,
  every geodesic segment can be covered with a finite number
of coordinate patches. The fact that $U$ and $\tilde\Omega$ are
star-shaped is a consequence of Eq.\ (\ref{onu}).

In order to prove the last statement of Theorem \ref{flo} it
suffices to  observe that we can restart the last argument beginning
with Eq.\ (\ref{bhq}) by introducing the Lipschitz dependence
$(x_0,\dot{x}_0)(z)$ where $z$ are local coordinates on $N$. Since
$(x,\dot{x})(t,z)$ is locally Lipschitz the whole argument stills
works where it is understood that $\p_z \varphi$ exists almost
everywhere in the Lebesgue $m$-dimensional measure of $N$ (This
measure can be equivalently defined either as done here using a
chart of $N$ or in a more intrinsic way regarding $N$ as a subset of
$TM$, see \cite[Sect.\ 3.3.3]{evans98}). Theorem \ref{flo} is
proved.

\begin{remark} \label{djh}
We pause for a moment to outline how to generalize this proof for
$C^{1,1}$ sprays over a $C^{3,1}$ manifold, the further generalization
to the $C^{k,1}$, $k\ge 0$, spray case being then analogous.

The idea is to introduce variables ${x}^\mu_{k,\beta}$,
${x}^\mu_{k,\dot{\beta}}$,  $\dot{x}^\mu_{k,\beta}$,
$\dot{x}^\mu_{k,\dot{\beta}}$  and add to the system
(\ref{p1})-(\ref{p2}) the equations obtained (formally)
differentiating the right-hand side (\ref{p1})-(\ref{p2}) with
respect to the initial conditions
\begin{align*}
x^\mu_{k+1,\beta}(t)&=\delta^\mu_\beta+\int_0^t \dot{x}_{k,\beta}^\mu(t)\, \dd t , \\
x^\mu_{k+1,\dot \beta}(t)&=\int_0^t \dot{x}_{k,\dot \beta}^\mu(t)\, \dd t , \\
\dot{x}^\mu_{k+1,\beta}(t)&= \int_0^t (\p_{x^\alpha} H^\mu)
x^\alpha_{k,\beta}+ (\p_{\dot{x}^\alpha} H^\mu) \dot
x^{\alpha}_{k,\beta} \,
\dd t,\\
\dot{x}^\mu_{k+1,\dot \beta}(t)&= \delta^\mu_\beta+ \int_0^t
(\p_{x^\alpha} H^\mu) x^\alpha_{k,\dot \beta}+ (\p_{\dot{x}^\alpha}
H^\mu) \dot x^{\alpha}_{k,\dot \beta} \, \dd t.
\end{align*}
The proof of the convergence for $k\to \infty$ proceeds as above.
From here it follows that ${x}^\mu_{,\beta}$,
${x}^\mu_{,\dot{\beta}}$,  $\dot{x}^\mu_{,\beta}$,
$\dot{x}^\mu_{,\dot{\beta}}$  are Lipschitz as they solve a first
order Lipschitz ODE. Since the solution is unique ${x}^\mu_{,\beta}$
coincides with $\p_{x_0^\beta} {x}^\mu$ and so on.
\end{remark}

\subsection{Strong differentiability of exp {\normalsize (Theorems \ref{jfd} and \ref{jfe})}}
Let us consider two solutions $x^\mu(t)$ and $y^\mu(t)$ of the
geodesic equation with initial conditions $(x^\mu_0,\dot x^\mu_0)$,
$(y^\mu_0,\dot y^\mu_0)$,  which belong to the domain given by Eqs.\
(\ref{d1}). Starting from these initial conditions we define
inductively functions $x_k^\mu(t)$, $y_k^\mu(t)$ as above, and
subtract the corresponding Eqs.\ (\ref{p1})-(\ref{p2}), rearranging
them as follows
\begin{align*}
x_{k+1}(t)- y_{k+1}(t)-(x_0-y_0)-(\dot{x}_0 -\dot{y}_0) t &=\int_0^t
[\dot{x}_k(t)-\dot{y}_{k}(t)-(\dot{x}_0-\dot{y}_0)
]\, \dd t ,\\
\dot{x}_{k+1}(t)-\dot{y}_{k+1}(t) -(\dot{x}_0-\dot{y}_0)
&=\!\int_0^t\!\![H(x_k(t),\dot{x}_k(t))-H(y_k(t),\dot{y}_k(t))]\,
\dd t .
\end{align*}
Thus
\begin{align*}
\Vert x_{k+1}(t)- y_{k+1}(t)-(x_0-y_0)-(\dot{x}_0 -\dot{y}_0) t
\Vert &\le \int_0^t  \Vert
\dot{x}_k(t)-\dot{y}_{k}(t)-(\dot{x}_0-\dot{y}_0) \Vert \, \dd t ,
\end{align*}
 and
\begin{align*}
\Vert \dot{x}_{k+1}(t)&-\dot{y}_{k+1}(t) -(\dot{x}_0-\dot{y}_0)
\Vert
\le \int_0^t \Vert H(x_k(t),\dot{x}_k(t))-H(y_k(t),\dot{y}_k(t))\Vert \, \dd t \\
&\le \int_0^t  \{A \Vert x_k -y_k\Vert +B \Vert \dot
x_k-\dot y_k \Vert \}\, \dd t \\
&\le \int_0^t  \{A [\Vert x_k -y_k-(x_0-y_0)-(\dot{x}_0 -\dot{y}_0)
t \Vert +\Vert x_0-y_0 \Vert+\Vert\dot{x}_0 -\dot{y}_0 \Vert
t]\\
&\qquad \ +B [\Vert \dot x_k-\dot y_k -(\dot{x}_0 -\dot{y}_0)
\Vert+\Vert\dot{x}_0 -\dot{y}_0 \Vert] \}\, \dd t.
\end{align*}
By induction it follows that
\begin{align*}
\Vert x_{k}(t)\!- \!y_{k}(t)\!-\!(x_0-y_0)\!-\!(\dot{x}_0
-\dot{y}_0) t \Vert & \le \max \{D \Vert x_0-y_0 \Vert, \Vert
\dot{x}_0-\dot{y}_0\Vert \} (\frac{e^{Dt}-1}{D}-t) ,\\
\Vert \dot{x}_{k}(t)-\dot{y}_{k}(t) -(\dot{x}_0-\dot{y}_0) \Vert
&\le \max \{D \Vert x_0-y_0 \Vert ,\Vert \dot{x}_0-\dot{y}_0\Vert \}
(e^{Dt}-1) .
\end{align*}
The induction hypothesis for the former equation and $k=1$ is satisfied because of Eq.\ (\ref{p1}).
The induction hypothesis for the latter equation and $k=1$ is satisfied because using Eq.\ (\ref{kiu})
\begin{align*}
\Vert \dot{x}_{1}(t)-\dot{y}_{1}(t) -(\dot{x}_0-\dot{y}_0) \Vert &
\le \int_0^t A \Vert x_0-y_0\Vert+B\Vert \dot x_0-\dot y_0\Vert \dd t\\
& \le \max \{D \Vert x_0-y_0 \Vert ,\Vert \dot{x}_0-\dot{y}_0\Vert \} Dt.
\end{align*}
Taking the limit $k\to\infty$ we obtain
\begin{align}
\Vert x(t)\!-\!y(t)\!-\!(x_0-y_0)\!-\!(\dot{x}_0 -\dot{y}_0) t \Vert
& \le \max \{ D \Vert x_0-y_0 \Vert , \Vert \dot{x}_0-\dot{y}_0\Vert
\} (\frac{e^{Dt}-1}{D}-t)
, \label{nik}\\
\Vert \dot{x}(t)-\dot{y}(t) -(\dot{x}_0-\dot{y}_0) \Vert &\le \max
\{D \Vert x_0-y_0 \Vert ,\Vert \dot{x}_0-\dot{y}_0\Vert \}
(e^{Dt}-1) , \label{nil}
\end{align}
Let us disregard for the moment the last inequality. We have the
trivial inequality
\begin{equation}
\Vert x_0- y_0-(x_0-y_0) \Vert  \le \max \{ D \Vert x_0-y_0 \Vert ,
\Vert \dot{x}_0-\dot{y}_0\Vert \} ({e^{Dt}-1}-Dt)
. \label{nic}\\
\end{equation}
On  $X=\mathbb{R}^{2n}$ let us consider a function $f: X\to X$
defined as follows
\[
f(x_0, \dot{x}_0)=(x_0, x(1)).
\]
Clearly, $f$ is the coordinate expression of the exponential map.
Let $L: X \to X$ be the linear map given by the matrix
\begin{equation} \label{nsd}
L=\begin{pmatrix} I &  0\\ I & I \end{pmatrix}
\end{equation}
with $I$ the $n\times n$ identity matrix. Recalling that $D\le 1$
the inequalities (\ref{nic})  and (\ref{nik})  can be rewritten for
$t=1$
\[
\Vert f(x_0,\dot x_0)-f(y_0,\dot y_0)-L((x_0,\dot x_0)-(y_0,\dot
y_0)) \Vert \le \Vert (x_0,\dot x_0)- (y_0,\dot y_0) \Vert
(\frac{e^{D}-1}{D}-1)
\]
for every $(x_0,\dot x_0)$ and $(y_0,\dot y_0)$ such that $\Vert
(x_0,\dot x_0)\Vert< \delta$ and $\Vert (y_0,\dot y_0)\Vert<
\delta$.

But Eq.\ (\ref{fun}) shows that $D$ as a function of $\delta$
satisfies $\lim_{\delta \to 0} D(\delta)=0$, thus $f$ is strongly
differentiable at $(0,0)$ with strong differential $L$. Since any
point $p\in M$ corresponds to zero coordinates for some chart
compatible with the atlas, we have that $\exp$ is differentiable at
any point in the image of $p\mapsto 0_p$.


Let us observe that if a strongly differentiable function $f(x,y)$
is strongly differentiable then keeping $x$ fixed we obtain a
function $f(x,\cdot)$ which is strongly differentiable. This fact
follows easily from the definition of strong differentiability.

Furthermore, the composition of strongly differentiable functions is
strongly differentiable. As a consequence,  the pointed exponential
map $\exp_p:=\pi_2(f(0, \cdot))$ is strongly differentiable at the
origin. From Eq.\ (\ref{nsd}) we read that the Jacobian is given by
the $n\times n$ identity matrix. Since it is invertible $\exp_p$
establishes a local Lipeomorphism.

\subsubsection{Normal vector bundle case {\normalsize (Theorem \ref{jfe})}}

Let us prove Theorem \ref{jfe}. Let $p\in \phi(S)$, let $\{s^k;
k=1,\cdots, l\}$ be coordinates on $S$ at $\phi^{-1}(p)$ and let
$x_0^\mu(s):=\phi^\mu(s)$. The tangent vectors $(\p_k
x_0^\mu)\p_\mu$, $k=1,\cdots, l$, provide a (Lipschitz) base for the
tangent space at any point in a neighborhood of $p$. Applying the
Gram-Schmidt procedure to
\[
 (\p_1 x_0^\mu)\p_\mu, \cdots ,  (\p_l
x_0^\mu)\p_\mu, \p_1, \cdots, \p_n,
\]
 discarding the last found
$l$ null vectors, and keeping the last $n-l$ non-trivial vectors,
we are left with a Lipschitz base of the normal space. Call this base
$e^\mu_k \p_\mu$, $k=1,\cdots, n-l$. By construction it is
Lipschitz. Thanks to this base we can introduce a chart of
coordinates $(s,y) \in \mathbb{R}^l\times \mathbb{R}^{n-l}$ over
$\nu(S)$, so as to represent each $v\in  \nu(S)$, $v=(x_0,\dot
x_0)$, as follows
\begin{align}
x_0^\mu(s,y)&=x^\mu_0(s^i), \\
 \dot{x}_0^\mu(s,y)&=  y^j e^\mu_j(s^i). \label{lrd}
\end{align}
This system of equations gives the map between $\nu(S)$ and $TM$
expressed through the respective coordinate charts.
 A naive calculation would
suggest that the Jacobian of this transformation is given by the
$(n+n) \times (l+(n-l))$ matrix
\[
J=
\begin{pmatrix}
\p_k x^\mu_0& 0\\
y^j \p_i e^\mu_j & e^\mu_j
\end{pmatrix}.
\]
However, since $e^\mu_j$ is just Lipschitz the $n\times l$ block
matrix $y^j \p_i e^\mu_j$ is not well-defined. Nevertheless, this
expression suggests that $J$ should be as given for $y=0$, namely on
$\phi(S)$, for in this case the ill defined block matrix vanishes.

Let us prove that for $y=0$, $J$ is the strong differential of the
map $(s,y) \to (x_0,\dot x_0)$. Let $\epsilon>0$. We need only to
show that for $y_1,y_2$ sufficiently close to zero and for $s_1,s_2$
sufficiently close to $s$
\begin{align*}
\Vert x_0(y_2,s_2)- x_0(y_1,s_1)-(\p_i x_0) (s_2^i-s_1^i)\Vert\ \le
\epsilon (\Vert
s_2-s_1\Vert+\Vert y_2-y_1\Vert  ) \\
\Vert \dot x_0(y_2,s_2)- \dot x_0(y_1,s_1)-e_j(s)
(y_2^j-y_1^j)\Vert\ \le \epsilon (\Vert s_2-s_1\Vert+\Vert
y_2-y_1\Vert  )
\end{align*}
The former inequality is a consequence of the fact that $x_0(s)$ is
$C^1$ hence strongly differentiable. The latter inequality can be
rewritten using Eq.\ (\ref{lrd})
\[
\Vert  y_2^j e_j(s_2)- y_1^j e_j(s_1)  -e_j(s) (y_2^j-y_1^j)\Vert\
\le \epsilon (\Vert s_2-s_1\Vert+\Vert y_2-y_1\Vert  )
\]
which is a consequence of the Lipschitzness of $e_j^\mu$.

The map
\[
(s,y) \to  (x_0,\dot x_0) \to \exp (x_0,\dot{x}_0) \to \pi_2(\exp
(x_0,\dot{x}_0) )
\]
being the composition of (1) a strongly differentiable map at $y=0$,
(2) a strongly differentiable map at $\dot{x}_0=0$, (3) the strongly
differentiable projection map $(x,y)\to y$, and being such that
 $\dot{x}_0(s,0)=0$,  is strongly
differentiable at $y=0$, with strong differential given by the $n \times n$ matrix $J_{\pi_2}
LJ=(\p_k x_0^\mu \, , \, e_j^\mu)(s)$ which is invertible because its
columns are linearly independent vectors. Thus the hypothesis of
Leach's inverse function theorem are satisfied.

\subsection{Convex neighborhoods {\normalsize (Theorem \ref{nsx})}} \label{vko} Let us prove that for sufficiently small
$\delta$, $\bar{B}(p,\delta)$ is reversible strictly convex normal.

By the strong differentiability of $\exp$ there
is an open neighborhood $O\ni (p,p)$, $O\subset M\times M$, which is
Lipeomorphic to an open neighborhood $U$ of $0_p\in TM$. Let
$\delta>0$ be sufficiently small that $\bar{B}(p,\delta)^2\subset O$ and
$\delta$ satisfies Eqs.\ (\ref{jux})-(\ref{juz}). Let $q\in
B(p,\delta)$, and let
\[
N_q=(\exp^{-1} {B}(p,\delta)^2 ) \cap \pi^{-1}(q)=[\exp^{-1}
(\{q\}\times {B}(p,\delta)) ] \cap \pi^{-1}(q)=\exp_q^{-1}
{B}(p,\delta) .
\]
The first defining equality shows that this set is
 open in the topology of $\pi^{-1}(q)=T_qM$.
Furthermore, $\exp_q\vert_{N_q}\colon N_q\to {B}(p,\delta)$ is
injective because $\{q\}\times N_q\subset U$ and so is
$\exp\vert_{U}$. It is surjective, for if $r\in {B}(p,\delta)$ then
$(q,r)\in  {B}(p,\delta)^2\subset O$, thus there is some $v\in
T_qM$, $v\in U$, such that $\exp_q v=r$. Finally,
$\exp_q\vert_{N_q}$ is Lipschitz because so is $\exp\vert_U$, and
$\exp^{-1}_q\vert_{{B}(p,\delta)}$ is Lipschitz because so is
$\exp^{-1}\vert_{{B}(p,\delta)^2}$. Thus we have proved that for
each $q\in B(p,\delta)$, there is an open set $N_q$ such that
$\exp_q: N_q\to B(p,\delta)$ is a Lipeomorphism. We stress that we
have not yet shown that $N_q$ is star-shaped.

Analogously, for sufficiently small $\delta$, for each $q\in
B(p,\delta)$, there is an open set $\tilde{N}_q(=\tilde{\exp}_q^{-1}
{B}(p,\delta))$ such that $\tilde{\exp}_q: \tilde{N}_q\to
B(p,\delta)$ is a Lipeomorphism.


Our choice of $\delta$ allows us to prove the strict convexity of
$\bar{B}(p,\delta)$ for both the spray and the reverse spray, and
hence that each $N_q$ and $\tilde{N}_q$ are star-shaped. The key
observation is that the continuous function defined on
$\pi^{-1}(\bar{B}(p,\delta))$ (``$\cdot$'' is the Euclidean scalar
product in $\mathbb{R}^n$ induced by the chart)
\[
z^\pm(x,v):=\Vert v\Vert^2+x \cdot H(x,\pm v)
\]
is positive if restricted to the unit tangent bundle
$\bar{B}(p,\delta)\times S^{n-1}$. Indeed,
\[
z^\pm(x,e):=1+x \cdot H(x,\pm e)\ge 1-\delta M >0
\]
where the last inequality is a consequence of Eq.\  (\ref{jux}). Let
us consider a geodesic segment contained in $\bar{B}(p,\delta)$. Any
geodesic $x(t)$ is $C^{2,1}$ thus $\Vert x \Vert^2(t)$ is $C^{2,1}$
and
\[
\frac{\dd^2 \Vert x \Vert^2}{ \dd t^2}= 2  (\Vert \dot x\Vert^2+ x
\cdot \frac{\dd^2 x }{ \dd t^2} ) =2 (\Vert \dot x\Vert^2+ x \cdot
H(x,\dot{x}) ) =2z^+(x,\frac{\dot x}{\Vert \dot x \Vert})\Vert \dot
x\Vert^2>0 ,
\]
where we used the fact that by definition a geodesic is regular,
i.e. $\dot{x}\ne 0$.

Analogously, if we consider a geodesic for the reverse spray
\[
\frac{\dd^2 \Vert x \Vert^2}{ \dd t^2}= 2  (\Vert \dot x\Vert^2+ x
\cdot \frac{\dd^2 x }{ \dd t^2} ) =2 (\Vert \dot x\Vert^2+ x \cdot
H(x,-\dot{x}) ) =2z^-(x,\frac{\dot x}{\Vert \dot x \Vert})\Vert \dot
x\Vert^2>0 .
\]
 As a consequence $\Vert x \Vert^2(t)$ takes its maximum value
at the boundary of its interval of definition, that is, in
correspondence of the endpoints of the geodesic segment which,
therefore, must be contained in   $\bar{B}(p,\delta)$. Furthermore,
if its endpoints are at the boundary of the ball then its interior
points stay in ${B}(p,\delta)$ because the inequality is strict.
Thus $C:=\bar{B}(p,\delta)$ is reversible strictly convex normal.
The fact that $\exp$ establishes a Lipeomorphism between an open
subset of $TC$ and $B(p,\delta)^2$ is immediate from the inclusion
$B(p,\delta)^2\subset O$. Analogously, the same result holds for
$\tilde{\exp}$. The proof of Theorem \ref{nsx} is complete.

\subsection{Role of the coordinate affine structure and position
vector} \label{aff}

We shall need the following result on the behavior of the position
vector on a convex neighborhood.
\begin{theorem} \label{nsc}
Under the assumptions of Theorem \ref{nsx}, for every $\epsilon>0$ we have for any sufficiently small $\delta$ that $C=B(p,\delta)$ not only satisfies the conclusions of Theorem \ref{nsx} but also that for every $q_1,q_2,q_1',q_2',q\in C$ we have,
interpreting the minus sign as that given by the affine structure
induced by the coordinate chart
\begin{align}
&\Vert [P(q'_1,q'_2)-(q'_2-q'_1)]-[P(q_1,q_2)-(q_2-q_1)]\Vert\le
\epsilon \max\{\Vert q_1'-q_1\Vert, \Vert q_2'-q_2\Vert\}, \nonumber \\
&\Vert P(q,q_2)-P(q,q_1)-(q_2-q_1)\Vert\le \epsilon \Vert
q_2-q_1\Vert,  \label{ksc} \\
&\Vert P(q_1,q_2)-(q_2-q_1)\Vert\le \epsilon \Vert q_2-q_1\Vert, \label{ksd}  \\
& \Vert P(q_1,q_2)\Vert \le \epsilon. \label{kse}
\end{align}
In the connection case, if the coordinate chart is chosen so as to
make $\Gamma^\mu_{\alpha \beta}(p)$ vanish, then $\delta$ can be
also chosen such that
\begin{equation}
\Vert P(q_1,q_2)+P(q_2,q_1)\Vert\le \epsilon \Vert q_2-q_1\Vert^2.
\label{ksf}
\end{equation}
\end{theorem}
This section will be devoted to the proof of this result.


Our coordinate system $\{x^\mu\}$ in a neighborhood $U$ of $p$
induces an affine structure which allows us to compare tangent
vectors at different points of $U$. Let us consider the geodesic $x(t)$ such that its initial condition $(x_0,\dot x_0)$ satisfy Eq.\ (\ref{d1}). We ask if, keeping $x_0$ fixed, the map $\dot{x}_0\mapsto
\dot{x}(1)$ is injective (observe that $\dot{x}(1)=P(x_0,x(1))=P(x_0,\exp_{x_0} \dot x_0$).

\begin{lemma} \label{hot}
Let $\Vert(x_0,\dot{x}_0)\Vert \le \delta$.  For fixed  base point
$x_0$, the map $\dot{x}_0\mapsto \dot{x}(1)$  is strongly differentiable at the origin,
the strong differential being the $n\times n$ identity matrix $I$.
The map $(x_0, \dot{x}_0)\mapsto (x_0,\dot{x}(1))$ is also strongly
differentiable wherever $\dot{x}_0=0$, the strong differential being
the matrix
\[
\begin{pmatrix}
I & 0 \\ 0& I
\end{pmatrix}.
\]
Thus for sufficiently small $\delta$ both maps are injective (and
bi-Lipschitz) with inverse strongly differentiable wherever
$\dot{x}(1)=0$.
\end{lemma}

\begin{proof}
Let us consider the exponential map of the vectors $\dot{x_0},
\dot{y}_0$ with base point $x_0$ (thus $y_0=x_0$). By Eq.\
(\ref{nil}) setting $t=1$
\[ \Vert \dot{x}(1)-\dot{y}(1)
-(\dot{x}_0-\dot{y}_0) \Vert \le \Vert \dot{x}_0-\dot{y}_0\Vert
(e^{D}-1) .
\]
Since $D(\delta)\to 0$ for $\delta \to 0$, the map $\dot{x}_0\mapsto
\dot{x}(1)$ is strongly differentiable at the origin, with strong
differential the identity matrix.

As for the map $(x_0, \dot{x}_0)\mapsto (x_0,\dot{x}(1))$ it
suffices to include in the previous analysis the trivial inequality
\[ \Vert x_0-y_0-(x_0-y_0)
 \Vert \le \Vert x_0-y_0\Vert
(e^{D}-1) ,
\]
and recall that on $\mathbb{R}^n\oplus \mathbb{R}^n$ we use the norm
$\max\{\Vert \, \Vert, \Vert \,\Vert\}$.
 The last claim follows from Leach's inverse function
theorem.
\end{proof}


\begin{proposition} \label{juy}
Let $p\in M$, and let $\{x^\mu\}$ be a chart in a neighborhood $U$
of $p$. Let $P(r,q)$ be the position vector of $q$ with respect to
$r$. There is an open convex  neighborhood $C\ni p$, such that for
every $q_2,q_1\in C$, the map $(q_1,q_2)\mapsto
P(q_1,q_2)-(q_2-q_1)$, interpreted with the affine structure induced
by $\{x^\mu\}$, is strongly differentiable on the diagonal
$q_1=q_2$, with zero strong differential.
\end{proposition}

\begin{proof}

Let $C$ be a convex neighborhood of $p$ such that $\exp$ establishes
a Lipeomorphism between an open subset of $TC$ and $C\times C$. Let
$x^\mu:C\mapsto \mathbb{R}^n$ be ($C^{2,1}$) coordinates on $C$,
such that $x^\mu(p)=0$.

We proved that $\exp$ is strongly differentiable on the zero section of $TC$ with differential $L$ on a suitable trivialization (Eq.\ (\ref{nsd})), thus by  Leach's inverse function theorem,   $\exp^{-1}$ is strongly differentiable on the diagonal of
$C\times C$ with differential $L^{-1}$. Stated in another way, the map in coordinates given by $(q_1,q_2) \mapsto
(q_1,\exp^{-1}_{q_1} q_2)$ is strongly differentiable on the
diagonal with differential $L^{-1}$.

Lemma \ref{hot} proves that the coordinate map $(q_1, v)\mapsto
P(q_1,\exp_{q_1} v)$ is strongly differentiable at the origin with
strong differential the identity, thus the coordinate map
$(q_1,q_2) \mapsto (q_1,P(q_1,q_2))$ is strongly differentiable on
the diagonal with strong differential
\[
L^{-1}=\begin{pmatrix} I&0\\-I& I \end{pmatrix}.
\]
 This is the same strong differential of the map $(q_1,q_2)
\mapsto (q_1, q_2-q_1)$ where the difference makes sense using the
affine structure induced by the coordinate chart. Thus the map
$(q_1,q_2)\mapsto (q_1,P(q_1,q_2)-(q_2-q_1))$ has vanishing strong
differential on the diagonal of $C\times C$ which implies that the
map $(q_1,q_2)\mapsto P(q_1,q_2)-(q_2-q_1)$ has vanishing strong
differential on the diagonal of $C\times C$.
\end{proof}

In particular the map $(q_1,q_2)\mapsto P(q_1,q_2)-(q_2-q_1)$ is
strongly differentiable at $(p,p)$ thus for every $\epsilon>0$ the
constant $\delta>0$ and hence the convex neighborhood
$C=B(p,\delta)$ can be chosen such that for every $q_1,q_2, q_1',
q_2'\in C$
\[
\Vert [P(q'_1,q'_2)-(q'_2-q'_1)]-[P(q_1,q_2)-(q_2-q_1)]\Vert\le
\epsilon \max\{\Vert q_1'-q_1\Vert, \Vert q_2'-q_2\Vert\}.
\]
Thus for every $q_1,q_2,q\in C$ (set $q_1'\to q$, $q_2'\to q_2$,
$q_2\to q_1$, $q_1\to q$)
\[
\Vert P(q,q_2)-P(q,q_1)-(q_2-q_1)\Vert\le \epsilon \Vert
q_2-q_1\Vert,
\]
and (set $q=q_1$)
\begin{equation} \label{usf}
\Vert P(q_1,q_2)-(q_2-q_1)\Vert\le \epsilon \Vert q_2-q_1\Vert.
\end{equation}
Thus $\Vert P(q_1,q_2)\Vert \le (1+\epsilon) \Vert q_2-q_1\Vert$ and
the diameter of $C$ can be chosen sufficiently small that $\Vert
P(q_1,q_2)\Vert\le \epsilon$.

Suppose now that the spray is a connection (hence reversible) and assume to have chosen
the coordinate system in such a way that for every  every normalized
vector $e$, $H^{\mu}(p,e)=0$. This is always possible through an
invertible quadratic coordinate change, so that the new coordinate
system is still $C^{2,1}$. By continuity $C$ can be chosen sufficiently small
that on $C$, $\Vert H\Vert:=\sup_{x\in C}\sup_{\Vert e\Vert=1} \Vert
H(x,e)\Vert \le \epsilon$.

Let $x:[0,1]\to C$, $x(0)=q_1'$, $x(1)=q_1$, be a geodesic. Let
$q_2,q_2'=x(t)$, $t\in [0,1]$ then the above 4-points inequality
gives
\[
\Vert P(x(0),x(t))-P(x(1),x(t))-(x(1) -x(0))\Vert\le \epsilon \Vert
x(1)-x(0)\Vert .
\]
We observe that $P(x(0),x(t))=t \dot{x}(t)$ and $P(x(1),x(t))=-(1-t)
\dot{x}(t)$ thus for every $t\in [0,1]$
\[
\Vert \dot{x}(t)-(x(1) -x(0))\Vert\le \epsilon \Vert x(1)-x(0)\Vert.
\]
We have $P(x(0),x(1))=\dot{x}(1)$ and
$P(x(1),x(0))=-\dot{x}(0)$ thus setting $r(t)=\dot{x}(t)-(x(1)
-x(0))$,
\begin{align*}
&\Vert P(x(0),x(1))+P(x(1),x(0))\Vert =\Vert
\dot{x}(1)-\dot{x}(0)\Vert
= \Vert \int_0^1 H(x(s),\dot{x}(s)) \dd s\Vert\\
&\le  \int_0^1 \Vert H(x(s),\hat \dot{x}(s)) \Vert \, \Vert \dot{x}(s)\Vert^2 \dd s\le \epsilon
\int_0^1 \Vert \dot{x}(s)\Vert^2 \dd s\\
&\le \epsilon [ \int_0^1 \Vert r(s)\Vert^2 \dd s+\int_0^1\Vert
x(1)-x(0)\Vert^2 \dd s+ 2 [x(1)-x(0)]\cdot \int_0^1 r(s)\dd s] \\
&\le \epsilon (1+\epsilon^2) \Vert x(1)-x(0)\Vert^2,
\end{align*}
thus a redefinition of $\epsilon$ gives the last inequality of
Theorem \ref{nsc}.

\subsection{Local Lipeomorphisms {\normalsize (Theorems \ref{jfd} and \ref{jfe})}}

Let $\varphi(v,t)=\gamma'_v(t)$ be the geodesic flow on $TM$. We
have, $\exp v=(\pi(v),\pi(\gamma_v(1)))$. If $\varphi(v,1)$ is well
defined then, by the continuity of the geodesic flow, so is
$\varphi(w,1)$ for $w$ near $v$. Thus $\Omega$ is open and,
analogously, so is $\Omega_p$.

The map $\exp$ is locally Lipschitz wherever it is defined because
$\pi$ is locally Lipschitz and if $v\in \Omega$, $\varphi(v,1)$ is
Lipschitz for $w$ near $v$ by the local Lipschitzness of the
geodesic flow (that is by the dependence on initial conditions of
solutions to the geodesic equation, see Sect.\ \ref{dep}). Analogously, $\exp_p$ is
locally Lipschitz on $\Omega_p$.

We have shown that for each $p$ there is a convex normal relatively compact open
neighborhood $C_p$, such that $\exp$ provides a Lipeomorphism
between the star-shaped relatively compact open set $\exp^{-1} C_p^2$ and $C_p^2$.

Let $\{\exp^{-1} C_{p_i}^2\}$ be a locally finite
covering of the image of the zero section $Z$ of $TM$, and let
$N=\cup_{p_i} \exp^{-1} C_{p_i}^2$.  Observe that if $w\in N$ then $w\in \exp^{-1} C_{p_i}^2$ for some $i$ and so it cannot be $\exp w=\pi(w)$ unless $w\in Z$ by the injectivity of $\exp$ on $C_p$.
By construction for
every compact set $K\subset M$, we have that $\pi^{-1}(K)\cap
\bar{N}$ is compact.

We want to show that there is an open subset
$E\subset N$, containing $Z$  such that $\exp\vert_E$ is injective
and hence a Lipeomorphism on its image. Let $K_i$, $K_i \subset
\textrm{Int} K_{i+1}\subset M$, be a sequence of compact sets such
that $\cup_i K_i=M$, and let $\bar{N_j}\subset N$, $\bar{N}_{j+1} \subset N_j$,
be open neighborhoods of $Z$ such that for every compact set $K$, $\pi^{-1}(K)\cap \bar N_j$ is  compact and $\cap_j \bar{N}_j=Z$. Clearly
$\exp$ is injective on $Z$.
 By induction, suppose that there is an increasing map $\sigma$ defined on $\{1,
 2, \cdots , k\}$ such that the map $\exp$ is
injective on the  set ($E_0:=Z$)
\[
E_k:=  Z\cup \bigcup_{i=1}^k \pi^{-1}(\textrm{Int}K_i)\cap {N}_{\sigma(i)}.
\]
and hence $E_j$ for $j\le k$.
Then we can define $\sigma(k+1)$ such that the same property holds
with $k$ replaced by $k+1$. For if not  we could find  sequence
$v_j, w_j \in E_k \cup (\pi^{-1}(\textrm{Int} K_{k+1}) \cap {N}_{j} )$, $v_j\ne w_j$, such that
for every $j$, $\exp v_j=\exp w_j$. On  the first component the last equality reads $\pi(v_j)=\pi(w_j)$. However, by injectivity of $\exp$ on $E_k$,  $v_j$, $w_j$ do not both belong to $E_k$ thus passing to a subsequence we can assume without loss of generality that
 $v_j  \in  \pi^{-1}(\textrm{Int} K_{k+1}) \cap {N}_{j}$.
Observe that  $v_j$ belongs to the compact set $\pi^{-1}(K_{k+1})\cap \bar{N}$,  and passing  to the limit we obtain
that up to subsequences $v_j$ converges to
\[
v \in \cap_j[\pi^{-1}( K_{k+1})\cap \bar{N}_j]\subset   \pi^{-1}(K_{k+1})\cap Z.
\]
In particular, $\pi(w_j)=\pi(v_j)$ converges to some point of $K_{k+1}$, thus for sufficiently large $j$, $w_j$ is contained in a compact set $\pi^{-1}(K_{k+2})\cap \bar{N}_1\subset N$ and so converges up to subsequences to some vector $w\in N$.
Using the continuity of $\pi$ we obtain $\pi(v)=\pi(w)$.
%
%
By the continuity of $\exp$, $\exp v=\exp w$, and using $v\in Z$, $\pi(v)=\exp w$.  Thus $\pi(w)=\exp w$, and since $w\in {N}$ we have by the observation above that $w\in Z$, thus $w=v=0_{\pi(v)}$.
 Let $C$ be a convex normal neighborhood of $r:=\pi(v)$,
($v=0_r$) then $\exp^{-1} C^2$ is a neighborhood of $v\in TC$ and the vectors
 $v_j$ and $w_j$ for sufficiently large $j$ enter it  which contradicts the injectivity of $\exp$ on $\exp^{-1} C^2$.
Finally,
$E=\cup_j E_j$ gives the searched open set.


The proof in the normal bundle case is analogous. We have to start
from a  sequence $p_i\in \phi(S)$ such that $C_{p_i}$ is
a locally finite covering of $\phi(S)$, define $N= \cup_{p_i}
\exp^{-1} C_{p_i}\subset \nu(S)$ and proceed as in \cite[Prop.\ 26,
Chap.\ 7 ]{oneill83} to prove the injectivity of $\exp_{\nu(S)}$ on
an open subset $E\subset N$.

%
%

%
%

\section{Proofs II: Pseudo-Finsler sprays and connections} \label{tud}

In the next section we prove Gauss' Lemma for sprays which come from
a pseudo-Finsler metric $L$.

\subsection{Gauss' Lemma {\normalsize (Theorem \ref{gux})}}
Let us  prove  Eq.\ (\ref{jwf}). Since $\exp_p^{-1}$ is Lipschitz,
$D^2_p$ is Lipschitz. By Theorem \ref{flo} the function $
2g_{P(p,q)}( P(p,q), \cdot) $ where $P(p,q):=\gamma_{\exp_p^{-1}
q}'(1)$, is Lipschitz in $q$. By Theorem \ref{pok} we need only to
show that the differential of $D^2_p$ exists and coincides with the
previous expression almost everywhere on $N$.

However, we know that $\exp$ is differentiable almost everywhere
over the star-shaped open set $\exp^{-1}N\subset T_pM$, and hence, by
Fubini's theorem, that it is almost everywhere differentiable on
almost every radial line passing through the origin (the expression
a.e.\ here refers to the ($n-1$)-dimensional Lebesgue measure of a
Euclidean sphere contained in $\exp^{-1}N$). It suffices to take $q$
on the exponential map of one of these geodesics (Lipeomorphisms
preserve zero measure sets \cite[Sect.\ 2.4]{evans98}).

Let $w\in T_qM$ and let $\sigma: [-a,a]\to N$ be a ($C^{2,1}$)
geodesic segment such that $\sigma'(0)=w$. Let $x^{(s)}: [0,1] \to N$,
$t\mapsto x^{(s)}(t)$, be the unique geodesic such that
$x^{(s)}(0)=p$, $x^{(s)}(1)=\sigma(s)$. Let $v(s)=\exp_p^{-1}
\sigma(s)$, since $\exp_p^{-1}$ is Lipschitz $v: [-a,a]\to T_pM$ is
Lipschitz. In particular $S=\{u\in T_pM: u=tv(s), t\in[0,1], s\in
[-a,a]\}$ is a Lipschitz submanifold of $T_pM$, thus by Theorem
\ref{flo} $x(t,s):=x^{(s)}(t)=\exp_p (tv(s))$ is Lipschitz  in
$(t,s)$. Furthermore, still by Theorem \ref{flo}
\begin{quote}
(*):  for almost every $s$, for  every $t\in [0,1]$,
$x(t,\cdot)$ is differentiable at $s$ and the derivative $\p_2
x(t,s)$ is locally Lipschitz in $t$, locally uniformly with respect
to those $s$ where it is defined. Finally, for any such $s$ we have
that for almost every $t\in [0,1]$, the mixed partial derivatives
$\p_1\p_2 x(t,s)$, $\p_2\p_1 x(t,s)$ are locally bounded and
coincide.
\end{quote}
The quantity $L(x^{(s)}(t),x^{(s)}_t(t))$ is independent of $t$
hence coincident with $L(p,v(s))$. The function $v(s)$ being
Lipschitz is differentiable almost everywhere. Let $s$ be such that
$v'(s)$ exists and (*) holds
\begin{align*}
\frac{1}{2}\,\p_{\sigma'(s)} D^2_p&= \frac{\dd }{\dd s}\,L(p,v(s))=
\frac{\p }{\p s}\, L(x(t,s),x_t(t,s)) = x^\mu_s  \frac{\p L}{\p
x^\mu}+ \frac{\p x^\mu_t}{\p s} \frac{\p L}{\p x^\mu_t}\\&=
\frac{\p}{\p t}(x^\mu_s \frac{\p L}{\p
x^\mu_t})-x^\mu_s(\frac{\p}{\p t} \frac{\p L}{\p x^\mu_t}-\frac{\p
L}{\p x^\mu})+(\frac{\p x^\mu_t}{\p s} -\frac{\p x^\mu_s}{\p
t})\frac{\p L}{\p x^\mu_t}
\end{align*}
where $t$ can be chosen arbitrarily. By (*) for almost every $t$ the
last term vanishes. Moreover, the second term of the right-hand side
vanishes because $x^{(s)}(t)$ is a geodesic, hence it solves the
Euler-Lagrange equation for the  Lagrangian $L$.
Integrating in $t$ over the interval $[0,1]$ we obtain, taking into
account that the left-hand side does not depend on $t$ and using
Eq.\ (\ref{njr})
\begin{align}
\frac{1}{2}\, \p_{\sigma'(s)} D^2_p&=
g_{x_t(1,s)}(x_t(1,s),\sigma'(s))-g_{x_t(0,s)}(x_t(0,s),x_s(0,s)) \nonumber\\
&=g_{\gamma'_{v(s)}(1)}(\gamma'_{v(s)}(1), \sigma'(s)), \label{mid}
\end{align}
where we used the fact that since $x(0,s)=p$, it is $x_s(0,s)=0$.
Evaluated at $s=0$ the previous expression proves Eq.\ (\ref{jwf})
and hence the first claim of the theorem.

Let $\alpha\colon [0,a)\to N$, $t \mapsto \alpha(t)$, be an integral
curve of $P$. Let us differentiate $D^2_p$ along it
\[
\frac{\dd D^2_p(\alpha(t))}{\dd t}=2 g_{P(\alpha(t))}(P(\alpha(t)),
\dot\alpha(t))=2g_{P(\alpha(t))}(P(\alpha(t)),
P(\alpha(t)))=2D^2_p(\alpha(t)).
\]
Then $D^2_p(\alpha(t))=D^2_p(\alpha(0)) \exp(2t)$, and since
$\exp^{-1} N$ is star-shaped the $-t$-time flow maps
$(D^2_p)^{-1}(s)$ to $(D^2_p)^{-1}(s e^{-2t})$ for $t>0$. Since $P$
is Lipschitz, by the mentioned result on the dependence of solutions
to first order ODE on the initial conditions,  the flow map is
Lipschitz, and since it is injective and can be inverted it is
actually a Lipeomorphism between $(D^2_p)^{-1}(s)$ and its image on
$(D^2_p)^{-1}(s e^{-2t})$.

Finally, suppose that $\exp_p$ is differentiable at $v\in T_p
M\backslash 0$. We observe that $\gamma'_v(1)=(d \exp_p)_v v$. Let $w\in T_v(T_pM)\sim T_pM$ and define $v(s):=v+s w$ and $\sigma(s):=\exp_p(v+s w)$, so that by the differentiability assumption $\sigma'(s)=(d \exp_p)_v w$. Due to the choice of curve $\sigma(s)$ we have $D^2_p(\sigma(s))=2L(p,\exp^{-1}_p(\sigma(s)))=2L(p,v+sw)$ thus from Eq.\ (\ref{njr})
\[
\frac{1}{2} \frac{\dd D^2_p(\sigma(s))}{\dd s}\vert_{s=0}=g_v(v,w)
\]
while from Eq.\ (\ref{mid})
\[
\frac{1}{2} \frac{\dd D^2_p(\sigma(s))}{\dd s}\vert_{s=0}=g_{(d \exp_p)_v v}((d \exp_p)_v v,(d \exp_p)_v w).
\]

\subsection{Local properties of geodesics in pseudo-Finsler geometry  {\normalsize (Theorem \ref{pwa})}}

Let $\sigma\colon [0,1] \to N$ be an AC-curve starting from $p$ and ending at $q\in N$.

Let us consider the Finsler case. By continuity  there is a last
value $\hat s\ge 0$, such that $D_p(\sigma(\hat{s}))=0$ (possibly
$\hat{s}=0$ or $\hat{s}=1$).

The  length of the geodesic connecting $p$ to $q$ is $D_p(q)$ and is
positive if and only if $q\ne p$. The statement of the theorem is
trivial for $q=p$ thus let us assume $q\ne p$ (so that $\hat{s}\ne
1$). Since $D_p^2$ is $C^{1,1}$, $D_p$ is $C^{1,1}$ in the region
$N\backslash\{p\}$. Thus $D_p(\sigma(s))$ being the composition of a
locally Lipschitz and an absolutely continuous function is
absolutely continuous. We have for $s\ge \hat{s}$
\begin{align*}
D_p(\sigma(s))&=\int_{\hat{s}}^s \!\frac{\dd D_p(\sigma(s))}{\dd s}
\,\dd s= \int_{\hat{s}}^s\!
\frac{1}{D_p(\sigma(s))}\,g_{P(p,\sigma(s))}(P(p,\sigma(s)),\sigma'(s)) \,\dd s\\
&= \!\!\int_{\hat{s}}^s \!\! g_{\hat{P}(p,\sigma(s))}(
\hat{P}(p,\sigma(s)),\sigma'(s))\, \dd s \le\!\! \int_{\hat{s}}^s \,
\!\!\!\sqrt{g_{\sigma'(s)}(\sigma'(s),\sigma'(s))} \,\dd s\le
l[\sigma],
\end{align*}
where $\hat P:=P/\sqrt{g_P(P,P)}$. In the last step we used the
analog of the Cauchy-Schwarz inequality for Finsler geometry
\cite[Theor. 1.2.2]{bao00}. For $s=1$ the above inequality proves
that the length of $\sigma$ is no smaller than that of the geodesic
connecting its endpoints. If they are equal then for almost every
$s\in [0,1]$, we have the equality $g_{\hat{P}(p,\sigma(s))}(
\hat{P}(p,\sigma(s)),\sigma'(s))= \sqrt{
g_{\sigma'(s)}(\sigma'(s),\sigma'(s))} $, thus, by the equality case
in the Finslerian Cauchy-Schwarz inequality we have that   for
almost every $s\in [0,1]$, $\sigma'\propto P(p,\sigma(s))$. If we introduce
spherical normal coordinates $(r,\theta_1,..., \theta_{n-1})$, as
this coordinate chart is Lipschitz related to those of $M$, $\sigma$
is still absolutely continuous in this chart. Thus since
$\sigma'\propto \p_r$ almost everywhere, the angular coordinates
cannot change over $\sigma$, otherwise since $\theta_i(\sigma(s))$ is the
integral of its own derivative one would get that $\sigma'$ is not
radial in a set of non-vanishing measure, a contradiction. Thus the
image of $\sigma$ coincides with the image of an integral curve of
$P$ and hence coincides with the image of the geodesic $\eta(r)$
connecting $p=\sigma(0)$ with $q=\sigma(1)$. Since the coordinates of
the spherical normal chart are Lipschitz functions, the composition
$r(s)$ is absolutely continuous. By definition  $r$ is an affine
parameter over the geodesic which has the same image of $\sigma$.
The map is necessarily increasing, for if $r(s_2)\le r(s_1)$ for
$s_1<s_2$, then we would have $r'< 0$ (by definition of AC-curve
$r'\ne 0$ almost everywhere) in a subset of measure different  from
zero on $[s_1,s_2]$, and it would be easy to obtain a shorter curve
cutting a piece of domain from $\sigma$, a contradiction to the
length minimization assumption.

Let us consider now the Lorentzian-Finsler case with $\sigma$ causal
and future directed. We recall that the analog to the reverse
Cauchy-Schwarz inequality for Finsler spacetimes reads
\cite{minguzzi13c}:
\begin{quote}
Let $v_1,v_2$ be causal and future directed then
\begin{equation} \label{nun}
-g_{v_1}({v}_1,{v}_2)\ge \sqrt{-g_{v_1}({v}_1,{v}_1)}\,
\sqrt{-g_{v_2}({v}_2,{v}_2)},
\end{equation}
with equality if and only if $v_1$ and $v_2$ are proportional.
\end{quote}

%

Suppose that for some $\tilde{s}$, $ D_p^2(\tilde{s})< 0$. The
Lorentzian-Finsler length of the geodesic connecting $p$ to $q$ is:
$D^L_p(q):=(- D^2_p(q))^{1/2}$. Since $D_p^2$ is $C^{1,1}$, $D^L_p$
is $C^{1,1}$ in the region $D_p^2< 0$. Thus $D^L(\sigma(s))$ being
the composition of a locally Lipschitz and an absolutely continuous
function is absolutely continuous. We know that $D_p^2(\tilde{s})<
0$ and by continuity the same inequality holds in an interval
$[\tilde{s},s]$ provided  $s$ is sufficiently close to $\tilde{s}$.
We have
\begin{align}
D_p^L(\sigma(s))&- D_p^L(\sigma(\tilde{s}))=\int_{\tilde{s}}^s
\frac{\dd D_p^L(\sigma(s))}{\dd s}\, \dd s \nonumber\\
& =-
\int_{\tilde{s}}^s
\frac{1}{D_p^L(\sigma(t))}\,g_{P(p,\sigma(s))}(P(p,\sigma(s)),\sigma'(s))\, \dd s \nonumber \\
&= -\int_{\tilde{s}}^s
\!\!g_{\hat{P}(p,\sigma(s))}(\hat{P}(p,\sigma(s)),\sigma'(s)) \, \dd s
\ge \int_{\hat{s}}^s \!\!\! \sqrt{-
g_{\sigma'(s)}(\sigma'(s),\sigma'(s))} \,\dd
s \nonumber\\
&\ge\, l[\sigma], \label{lol}
\end{align}
where $\hat P:=P/\sqrt{-g_P(P,P)}$. In the last inequality we used
the above Finslerian reverse Cauchy-Schwarz inequality.

The equality so obtained proves that once $\sigma$ enters a region
with $D^2_p<0$ (the chronological future of $p$) it remains in that
region.

Now let $\eta\colon [-\epsilon,0] \to N$, $\eta(0)=p$, be a small
future directed timelike geodesic contained in a reversible convex normal neighborhood
of $p$, $C\subset N$. For sufficiently small $s$, $\sigma(s)\in C$,
and the curve obtained concatenating $\eta$ with $\sigma$ which
connects $\eta(-\epsilon)$ to $\sigma(s)$ starts  with a timelike
geodesic, hence it enters the chronological future of
$\eta(-\epsilon)$, and hence, by the above argument there is a
future directed timelike geodesic $\nu^{(\epsilon)}$ connecting
$\eta(-\epsilon)$ with $\sigma(s)$. Letting $\epsilon \to 0$, and
using the continuity of the exponential map $\tilde{\exp}$ for the reverse spray at $\sigma(s)$ we infer
the existence of a geodesic connecting $p$ to $\sigma(s)$, which by
the continuity of $g_{v}(v,v)$ at $ T_{\sigma(s)}M$ must be future
directed causal. As $s$ is arbitrary we have shown that in a maximal
closed interval $[0,b]\subset [0,1]$, $b>0$, we have
 $D^2_p(\sigma(s))\le 0$.

Let us prove that if for $a\in (0,b]$, $D^2_p(\sigma(a))=0$ then
$\sigma\vert_{[0,a]}$ is a lightlike geodesic up to parametrizations
and hence that $D^2_p=0$ over $[0,a]$.

Observe that $D^2_p$ is Lipschitz thus $D^2_p(\sigma(s))$ is
absolutely continuous
\[
D^2_p(\sigma(a))=\int_0^a \frac{\dd D^2_p(\sigma(s))}{\dd s}\, \dd
s=2 \int_0^a g_{P(p,\sigma(s))}(P(p,\sigma(s)), \sigma'(s))\, \dd s.
\]
Since on the region $D^2_p\le 0$, we have
$g_{P(p,\sigma(s))}(P(p,\sigma(s)), \sigma'(s))\le 0$ for almost every $s$ (by the Finslerian reverse Cauchy-Schwarz inequality since $\sigma'$ is future directed causal almost everywhere), thus we can have $D^2_p(\sigma(a))=0$
only if $\sigma'\propto P$ for almost every $s$ in $[0,a]$.
Introducing a Euclidean scalar product on $T_pM$, associated
spherical normal coordinates over $N$, and arguing as above for the
Finsler case we obtain that $\sigma\vert_{[0,a]}$ is an integral
curve of $P$, hence a lightlike geodesic issued from $p$.

From now on
let $a$ be the maximum value of $s$ for which $D^2_p(\sigma(s))=0$.

It remains only to prove that $b=1$. Suppose not then $a=b$
otherwise $D^2_p(b)<0$ which would imply the same inequality also in
$(b,1]$, a contradiction to $b<1$. Set $p'=\sigma(b)$ and take a
reversible convex normal neighborhood $C'\ni p'$,  $C'\subset N$. Arguing
as above proves that for any sufficiently small $\delta$, $p'$ is
connected to $\sigma(b+\delta)$ by a future directed causal geodesic
$\eta: [0,1]\to C'$. This geodesic cannot be the prolongation of the
lightlike geodesic $\sigma_{[0,b]}$ for we would get
$D_p^2(\sigma(b+\alpha\delta))\le 0$, $\alpha\in [0,1]$,  a
contradiction to the maximality of $b$. Thus the scalar product
$g_{P(p,\eta(t))}(P(p,\eta(t)), \eta'(t))$ is negative for $t=0$ and
hence in a neighborhood of $t=0$.  Now observe that $D^2_p$ is
Lipschitz thus $D^2_p(\eta(t))$ is absolutely continuous, and for
sufficiently small $t$
\[
D^2_p(\eta(t))=D^2_p(\sigma(b))+\int_0^t\! \frac{\dd
D^2_p(\eta(t))}{\dd t}\, \dd t=2\! \int_0^t\!\!
g_{P(p,\eta(t))}(P(p,\eta(t)), \eta'(t))\, \dd t<0.
\]
As the concatenation of $\sigma_{\vert_{[0,b]}}$ with $\eta$ is a
causal AC-curve and on it $D^2_p$ becomes negative at some point,
and it remains so, we have at the endpoint
$D^2_p(\sigma(b+\delta))=D^2_p(\eta(1))<0$. As $\delta$ is
arbitrary we get a contradiction to the maximality of $b$. The
contradiction proves that $b=1$.

If $\sigma$ is a lightlike geodesic up to parametrization, then
clearly its Lorentzian-Finsler length vanishes and the inequality
$D_p^L(\sigma(1))\ge l(\sigma)$ is satisfied. Suppose that $\sigma$ is
not a lightlike geodesic up to parametrizations then $a<1$, and its
Lorentzian-Finsler length is given just by the contribution of
$\sigma_{[a,1]}$. Let $\tilde{s}\in [a,1]$ so that
$D^2_p(\sigma(\tilde{s}))<0$. By (\ref{lol})
\[
D^L_p(\sigma(1))\ge l(\sigma_{[\tilde{s},1]})
\]
and taking the limit $\tilde{s}\to a$ we obtain $D^L_p(\sigma(1))\ge
l(\sigma)$. This proves that $\sigma$ has a Lorentzian-Finsler
length no larger than that of the  geodesic connecting its
endpoints.

Now, suppose by contradiction that they have the same Lorentzian-Finsler
length and that $\sigma$ is not a causal geodesic up to
parametrizations. Then necessarily $a<1$, for otherwise it would be
a lightlike geodesic. But then from (\ref{lol}), for $\tilde{s}>a$,
\[
D^L_p(\sigma(1))\ge D^L_p(\sigma(\tilde{s})+
l(\sigma_{[\tilde{s},1]})\ge
l(\sigma_{[0,\tilde{s}]})+l(\sigma_{[\tilde{s},1]})=l(\sigma).
\]
Thus the equality implies that the first inequality is actually an
equality which implies that
$g_{P(p,\sigma(s))}(P(p,\sigma(s)),\sigma'(s))=0$ for almost every
$s\in [\tilde{s}, 1]$, and hence, by the arbitrariness of
$\tilde{s}$, $\sigma'\propto P$ for almost every $s\in [a,1]$.
Introducing again spherical normal coordinates and arguing as above
proves that the image of $\sigma\vert_{[a,1]}$ is an integral curve
of $P$  (and hence the prolongation of
$\sigma\vert_{[0,a]}$ if $a\ne 0$) thus it is the image of a
geodesic.

Finally, suppose that the image of $\sigma$ coincides with that of a
causal geodesic $\eta$. Since the coordinates of the spherical
normal chart are Lipschitz functions, the composition $r(s)$ is
absolutely continuous. By definition  $r$ is an affine parameter
over the geodesic $\eta$. The map $r(s)$ is necessarily increasing,
for if $r(s_2)\le r(s_1)$ for $s_1<s_2$, then we would have $r'\le
0$ and hence in a subset of measure different  from zero on
$[s_1,s_2]$, which would imply that $\frac{\dd }{\dd s }
\sigma=(\frac{\dd }{\dd r } \eta) r'$ is not future directed causal in a set of
measure different from zero, a contradiction to the definition of future directed
causal AC-curve.

\subsection{Strong convexity of squared Riemannian distance {\normalsize (Theorem \ref{keg})}}

Let $C$ be a convex neighborhood of $p$ as in Theorems \ref{nsx} and
\ref{nsc} where $\epsilon\in (0,1)$ and where the coordinate chart is chosen so that $\Gamma^\mu_{\alpha \beta}(p)=0$.

 Let us consider the Riemannian case. From Eq.\ (\ref{ksc}) using the Cauchy-Schwarz
inequality
\begin{equation}
\vert P(q,q_2)\cdot (q_2-q_1)-P(q,q_1)\cdot
(q_2-q_1)-(q_2-q_1)^2\vert\le \epsilon \Vert q_2-q_1\Vert^2.
\end{equation}
Let $g_{q}$ be the matrix of $g$ at $q$. Since $g$ is $C^{1}$ it is
strongly differentiable and by the choice of coordinate system its
strong derivative vanishes at $p$. Thus $C$ can be chosen sufficiently small
that for every $q_1,q_2\in C$
\[
\Vert g_{q_2}-g_{q_1}\Vert \le \epsilon \Vert q_2-q_1\Vert
\]
Moreover, $C$ can be chosen sufficiently small that once expressed in the
coordinate chart $\Vert g-I\Vert \le \epsilon$ on $C$.


Thus if $D^2_{q}$ is the squared distance function from $q$, using
also Eqs. (\ref{ksc}), (\ref{ksd}) and (\ref{kse}) we are able to
prove Eq.\ (\ref{coa})
\begin{align*}
& \frac{1}{2}\vert [\dd D^2_q(q_2)- \dd D^2_q(q_1)] (q_2-q_1)-2(q_2-q_1)^2
\vert \\
&=\vert g_{q_2}(P(q,q_2), (q_2-q_1))-g_{q_1}(P(q,q_1),
(q_2-q_1))-(q_2-q_1)^2\vert\\
&=\vert g_{q_2}(P(q,q_2)-P(q,q_1),
(q_2-q_1))-(g_{q_1}-g_{q_2})(P(q,q_1), (q_2-q_1))-(q_2-q_1)^2\vert
\\
&\le\vert [P(q,q_2)-P(q,q_1)]\cdot (q_2-q_1)-(q_2-q_1)^2\vert+\epsilon
\Vert P(q,q_2)-P(q,q_1)
\Vert\, \Vert q_2-q_1\Vert \\
&\quad+ \epsilon \Vert P(q,q_2)\Vert\, \Vert q_2-q_1\Vert^2\\
 &  \le  2\epsilon (1+ \epsilon) \Vert
q_2-q_1\Vert^2,
\end{align*}
which proves that $D^2_{q}$ is strongly convex with respect to the
affine structure induced by the coordinate chart.

Let us give a geodesic version. This time we shall
need to use Eq.\ (\ref{ksf}). Also observe that from Eq.\
(\ref{ksd}) we have
\begin{align*}
\Vert P(q_1,q_2)-(q_2-q_1)\Vert&\le \epsilon \Vert q_2-q_1\Vert\le
\epsilon \Vert P(q_1,q_2)-(q_2-q_1)\Vert+\epsilon \Vert
P(q_1,q_2)\Vert,
\end{align*}
and hence
\[
\Vert P(q_1,q_2)-(q_2-q_1)\Vert\le \frac{\epsilon}{1-\epsilon} \Vert
P(q_1,q_2)\Vert,
\]
and
\begin{equation} \label{upo}
\Vert q_2-q_1\Vert\le \frac{1}{1-\epsilon} \Vert P(q_1,q_2)\Vert.
\end{equation}
Similarly, let $g$ be a metric such that at any point of $C$, $\Vert
g-I\Vert \le \epsilon$, and let $v$ be any vector. Then
\[
\vert (g-I)(v,v)\vert \le \epsilon v\cdot v = \epsilon
\vert(I-g)(v,v)+g(v,v) \vert\le  \epsilon \vert  (g-I)(v,v)\vert+\epsilon g(v,v),
\]
from which we obtain
\[
\vert (g-I)(v,v)\vert \le \frac{\epsilon}{1-\epsilon} \,\vert
g(v,v)\vert ,
\]
and
\[
\vert v\cdot v\vert\le \vert (I-g)(v,v)\vert +\vert g(v,v)\vert\le
\frac{1}{1-\epsilon} \,\vert g(v,v)\vert.
\]

Let $x:[0,1]\to C$ be a geodesic and let $q_1:=x(0)$, $q_2:=x(1)$.
We are  ready to prove Eq.\ (\ref{cob}).
\begin{align*}
&\frac{1}{2} \vert \frac{\dd }{\dd t}D(q,x(t))^2\vert_{t=1}-\frac{\dd }{\dd t} D(q,x(t))^2\vert_{t=0}-2D(q_1,q_2)^2 \vert\\
 &=\vert
g_{q_2}(P(q,q_2),P(q_1,q_2))+g_{q_1}(P(q,q_1),P(q_2,q_1))-
g_{q_1}(P(q_2,q_1),P(q_2,q_1)) \vert\\
&\le \vert
g_{q_1}(P(q,q_2),P(q_1,q_2))+g_{q_1}(P(q,q_1),P(q_2,q_1))-
g_{q_1}(P(q_2,q_1),P(q_2,q_1)) \vert \\& \quad +
\epsilon^2(1+\epsilon) \Vert q_2-q_1
\Vert^2 \\
&\le \vert
g_{q_1}(P(q,q_2),-P(q_2,q_1))+g_{q_1}(P(q,q_1),P(q_2,q_1))-
g_{q_1}(P(q_2,q_1),P(q_2,q_1)) \vert \\& \quad +2(1+\epsilon)
\epsilon^2 \Vert q_2-q_1 \Vert^2 \\
&\le \vert g_{q_1}(P(q,q_1)-P(q,q_2),P(q_2,q_1))-
g_{q_1}(q_1-q_2,P(q_2,q_1)) \vert \\& \quad +(2(1+\epsilon)
\epsilon^2+\epsilon(1+\epsilon)^2 ) \Vert q_2-q_1 \Vert^2 \\
&\le \vert g_{q_1}(P(q,q_1)-P(q,q_2)-(q_1-q_2),P(q_2,q_1))\vert
\\& \quad +(2(1+\epsilon)
\epsilon^2+\epsilon(1+\epsilon)^2 ) \Vert q_2-q_1 \Vert^2 \\
& \le 2 \epsilon (1+\epsilon) (1+2\epsilon)  \Vert q_2-q_1
\Vert^2\le 2 \epsilon \frac{(1+\epsilon) (1+2\epsilon)
}{(1-\epsilon)^2}\, \Vert P(q_1,q_2) \Vert^2\\ & \le 2 \epsilon
\frac{(1+\epsilon) (1+2\epsilon) }{(1-\epsilon)^3}
\,g_{q_2}(P(q_1,q_2),P(q_1,q_2)) \le 2 \epsilon \frac{(1+\epsilon)
(1+2\epsilon) }{(1-\epsilon)^3} \,D(q_1,q_2)^2.
\end{align*}
A reparametrization of $x$ with arc-length and a redefinition of
$\epsilon$ gives Eq.\ (\ref{cob}).

The statement of Theorem \ref{keg} concerning the strong convexity
of $D^2_q$ is immediate from the triangle inequality and from the
equivalences recalled in Sect.\ \ref{nof}.

\subsection{Splitting of the metric at a given radius  {\normalsize (Theorem \ref{oii})}}

On the coordinate ball let us introduce radial coordinates
$(\rho,\theta_1,\cdots,\theta_{n-1})$. Each  ball
$(D_p)^{-1}([0,r])$ is convex with respect to the affine structure induced by the
coordinate chart $(x_1,\cdots,x_n)$, thus the radial lines issued from
$p$ intersect the boundary of the ball only once. Let $r=D_p$, there
is therefore a function $\rho(r,\theta)$ establishing the dependence
of the radial coordinate on the angular ones. We known that
$r(q(\rho,\theta))$ is $C^{1,1}$, and Eq.\ (\ref{ksd}) and $\Vert
g-I\Vert \le \epsilon$ imply that $g(P(p,q), (q-p))\ne 0$, namely
$\p r/\p \rho \ne 0$. By the usual implicit function theorem
$\rho(r,\theta)$ is $C^{1,1}$, thus the components of $g$ in
coordinates $(r,\theta_1,\cdots,\theta_{n-1})$ are locally
Lipschitz. In particular, for any given $r>0$ the map which sends
$S^{n-1}$ (i.e.\ $\theta$) to $D_p^{-1}(r)$ is $C^{1,1}$ thus
differentiable.

Taking $r=cnst.$ in $g_{ij}(r,\theta)$ shows that the metric
induced on each hypersurface $D_p^{-1}(r)$ is locally Lipschitz.

The function $r$ is $C^{1,1}$ and by Theorem \ref{gux} $ \nabla r$
is the normalized geodesic field orthogonal to the level sets of
constant $r$. Thus $g^{-1}(\dd r,\dd r)=g(\nabla r, \nabla r)=1$ and
Eq.\ (\ref{kiv})
\[
g^{-1}=(\p_r+A_i(r,\theta)\p_i)^2+(h_r^{-1})_{ij} \p_i\otimes \p_j
\]
holds for some Lipschitz components $A_i$, $(h_r^{-1})_{ij}$. Its
inverse is
\[
g=d r^2+(h_r)_{ij}( \dd \theta_i-A_i \dd r) ( \dd \theta_j-A_j \dd
r),
\]
thus $h_r$ is the metric induced on the level set $D_p^{-1}(r)$.

Let us fix $\bar{r}$ so that $D_p^{-1}(\bar{r}) \subset C$. We know
that $\{\theta_i\}$ provides a $C^{1,1}$ chart over
$D_p^{-1}(\bar{r})$. Let us rename these coordinates $\{\alpha_i\}$
and let us extend them in a neighborhood of $D_p^{-1}(\bar{r})$
solving the differential equation $\p \theta_i/\p r=A_i(r,\theta)$,
$\theta_i(\bar{r},\alpha)=\alpha_i$. Since $A_i$ is Lipschitz the
solution $\theta(r,\alpha)$ is Lipschitz in $\alpha$  and $C^{1,1}$
in $r$ \cite[Ex.\ 1.2, Chap.\ 2]{hartman64} \cite[Prop.\
1.10.1]{cartan71} \cite[Cor.\ 1.6]{lang95}. Thus $g$ can be brought
to a direct sum form almost everywhere where the Jacobian $\p
\theta/\p \alpha$ exists. This Jacobian exists at $r=\bar{r}$ and
equals the identity matrix, thus
\[
h'_{ij}(\bar{r},\alpha)= h_{ks}(\bar{r},\theta(\bar{r},\alpha))
J^{k}_{i} J^{s}_{j}=
h_{ij}(\bar{r},\theta(\bar{r},\alpha))=h_{ij}(\bar{r},\alpha).
\]
This equality proves that the components $h'_{ij}(\bar{r},\cdot)$
exist and are Lipschitz in $\alpha$.

Of course, since the geodesic flow is Lipschitz we could deduce
immediately that normal spherical coordinates are Lipschitz and
hence that  $g$ can be brought to a direct sum form almost
everywhere, but we wanted to construct a coordinate system for which
the direct sum form was valid everywhere at a given radius.

\subsection{Some local results on the strong concavity of the  squared Lorentzian distance {\normalsize (Theorems \ref{igy}, \ref{jse}, Corollary \ref{cop})}}

\begin{proof}[Proof of Theorem \ref{igy}]
Let $C$ be a convex neighborhood of $p$ as in Theorems \ref{nsx} and
\ref{nsc} where $\epsilon\in (0,1/3)$ and the coordinate system is
such that $g_{\alpha \beta}(p)=\eta_{\alpha \beta}$,
$\Gamma^\mu_{\alpha \beta}(p)=0$. The metric defined by
\begin{equation} \label{iud}
r(v_1,v_2)=g(v_1,v_2)+2 (\p_0 \cdot v_1) (\p_0\cdot v_2)
\end{equation}
is positive definite at $p$ and hence $C$ can be chosen sufficiently small
that it is positive definite everywhere in $C$. Observe that at $p$
the metric $r$ once expressed in components coincides with the
identity matrix. Thus $C$ can be chosen sufficiently small that
\[
\Vert r-I\Vert \le \epsilon.
\]
From Eqs.\ (\ref{ksc})-(\ref{ksd})
\begin{equation}
 \Vert
P(q,q_1)-P(q,q_2)-P(q_2,q_1)\Vert\le 2\epsilon \Vert q_2-q_1\Vert.
\end{equation}
Let $x:[0,1]\to C$ be a geodesic, let  $q_1:=x(0)$ and $q_2:=x(1)$,
and let 
\[V=P(q,q_1)-P(q,q_2)-P(q_2,q_1).\]
We have
\begin{align*}
\vert g_{q_1}(V,P(q_2,q_1))\vert&=\vert r(V,P(q_2,q_1))- 2(\p_0\cdot
 V) (\p_0\cdot P(q_2,q_1))\vert\\
& \le \Vert V\Vert\, \Vert P(q_2,q_1)\Vert\, [\Vert r\Vert +2\vert
]\le \Vert V\Vert\, \Vert P(q_2,q_1)\Vert\, (3+\epsilon)\\
&\le 2 \epsilon (1+\epsilon)(3+\epsilon) \Vert q_2-q_1\Vert^2
\end{align*}
Since $g_q$ is $C^{1,1}$ in $q$ it is strongly differentiable at $p$
with zero strong differential, thus $C$ can be chosen sufficiently small that
for every $q_1,q_2\in C$,
\begin{equation} \label{jug}
\Vert g_{q_1}\Vert\le 1+\epsilon, \qquad \Vert
g_{q_2}-g_{q_1}\Vert\le \epsilon \Vert q_2-q_1\Vert.
\end{equation}
Thus
\begin{align*}
&\frac{1}{2} \vert \frac{\dd }{\dd t}D^2_q(x(t))\vert_{t=1}-\frac{\dd }{\dd t} D^2_q(x(t))\vert_{t=0}-2D_{q_1}^2(q_2) \vert\\
 &=\vert
g_{q_2}(P(q,q_2),P(q_1,q_2))+g_{q_1}(P(q,q_1),P(q_2,q_1))-
g_{q_1}(P(q_2,q_1),P(q_2,q_1)) \vert\\
&\le \vert
g_{q_1}(P(q,q_2),P(q_1,q_2))+g_{q_1}(P(q,q_1),P(q_2,q_1))-
g_{q_1}(P(q_2,q_1),P(q_2,q_1)) \vert \\& \quad +
\epsilon^2(1+\epsilon) \Vert q_2-q_1
\Vert^2 \\
&\le \vert
g_{q_1}(P(q,q_2),-P(q_2,q_1))+g_{q_1}(P(q,q_1),P(q_2,q_1))-
g_{q_1}(P(q_2,q_1),P(q_2,q_1)) \vert \\& \quad +2(1+\epsilon)
\epsilon^2 \Vert q_2-q_1 \Vert^2
\\
&\le \vert g_{q_1}(V,P(q_2,q_1))\vert +2(1+\epsilon) \epsilon^2
\Vert q_2-q_1 \Vert^2
\\
&\le 2 \epsilon (1+\epsilon)(3+2\epsilon) \Vert q_2-q_1 \Vert^2
\end{align*}
A redefinition of $\epsilon$ proves Eq.\ (\ref{pat}).
\end{proof}

%
%
%

\begin{proof}[Proof of Theorem \ref{jse}] It is sufficient to prove Eq.\ (\ref{lor}) for $\epsilon< 1/9$. Let
us parametrize $\gamma$ with respect to $g$-arc length (i.e.\ proper
time), $g(\dot \gamma,\dot \gamma)=-1$. As a first step let us
introduce, through a quadratic locally invertible coordinate transformation, a coordinate system such that $\dot{\gamma}(0)=\p_0$,
$g_{\alpha \beta}(p)=\eta_{\alpha \beta}$ and $\Gamma^\mu_{\alpha
\beta}(p)=0$, so that  the hypothesis of Theorem \ref{igy} apply.

Let us consider the spacelike subspace of $T_{\gamma(t)} M$, given
by
\[
S(t):=\textrm{Ker} \, g(\dot{\gamma}(t),\cdot).
\]
Let $L(t)$ be the subspace of $T_{\gamma(t)}M$ spanned by $\{\p_i,
i\ge 1\}$. For $t=0$, $S(0)=L(0)$, thus by
 continuity there is a neighborhood of $0$,
such that for $t$ belonging to this neighborhood $S(t)$ makes with
$L(t)$ an (Euclidean) angle smaller than $1/16$ rad. Thus $C$ can be taken sufficiently small that this property holds for every $t\in \gamma^{-1}(C)$.

Let $x:[0,1]\to C$ be a geodesic such that $x(0)$ and $x(1)$ belong
to the same level set of $(D^2_q)^{-1}(c)$, $c<0$, for some $q\in
C$. By Theorem \ref{pwa} $x$ cannot be future directed causal
otherwise $D^2_q(x(1))>D^2_q(x(0))$, and it cannot be past directed
causal otherwise $D^2_q(x(1))<D^2_q(x(0))$, thus $x$ is spacelike.
Let $a,b\in [0,1]$, $a<b$, then $y(t)=x((b-a)t+a)$ is such that
$y(0)=x(a)$ and $y(1)=x(b)$.
 Let us use Eq.\ (\ref{upo})   in
Eq.\ (\ref{pat}) for the geodesic $y$
\begin{equation} \label{par}
\vert \frac{\dd }{\dd t}D^2_q(y(t))\vert_{t=1}-\frac{\dd }{\dd t}
D^2_q(y(t))\vert_{t=0}-2D^2(x(a),x(b)) \vert\le
\frac{\epsilon}{(1-\epsilon)^2}\,\Vert P(x(a),x(b)) \Vert^2.
\end{equation}
Now we have to impose some constraint on $x(a)$ and $x(b)$ so as to
obtain an inequality of the form
\[\Vert P(x(a),x(b))
\Vert^2\le {\sigma(\epsilon)} g_{x(b)}(P(x(a),x(b)),P(x(a),x(b)))
\]
where $\epsilon \sigma(\epsilon) \to 0$ for $\epsilon \to 0$.

We recall that given two Lorentzian metrics $g_1,g_2$ on a
differentiable manifold, $g_1<g_2$ means that at each point the
timelike cone of $g_2$ contains the causal cone of $g_1$. Let me
consider the metric
\[
\eta^+:=-\frac{1+2\sqrt{\epsilon}}{1-2 \sqrt{\epsilon}}\, (\dd
x^0)^2+(\dd \vec{x})^2
\]
which satisfies $g<\eta^+$ at $p$. By continuity we can choose $C$
so small that it holds anywhere in $C$.

Suppose that $v\in TC$ is a spacelike vector for $\eta^+$, that is
$\eta^+(v,v)\ge 0$, then a little algebra shows that this condition
can be rewritten
\[
v\cdot v\le \frac{1}{\sqrt{\epsilon}}(-(1+\sqrt{\epsilon})(
v^0)^2+(1-\sqrt{\epsilon})( \vec{v})^2).
\]
Now observe that at $p$ for every $v\in T_p C$, $\Vert v\Vert=1$,
\[
-(1+\sqrt{\epsilon})( v^0)^2+(1-\sqrt{\epsilon})( \vec{v})^2<
g_p(v,v)=-(v^0)^2+(\vec{v})^2,
\]
thus by continuity the same holds  at any point in a neighborhood of
$p$, and we can choose $C$ sufficiently small that for every $v\in
TC$, and $q\in C$
\[
(-(1+\sqrt{\epsilon})( v^0)^2+(1-\sqrt{\epsilon})( \vec{v})^2)\le
g_q(v,v).
\]
As a consequence, for every $q\in C$ and for every $v\in TC$ such that $\eta^+(v,v)\ge 0$
\begin{equation} \label{jwj}
v\cdot v\le
\frac{1}{\sqrt{\epsilon}}\, g_q(v,v).
\end{equation}


Let  $q=\gamma(t_q)$, $r=\gamma(t_r)$, $t_q,t_r\in I$, $t_q\ne t_r$.

In Section \ref{vko} we proved that $O:=\bar{B}(r,\delta)$ is
strictly convex normal for any sufficiently small $\delta$, and  in
Section \ref{oyt} through Eq.\ (\ref{pki}) we proved that the
Euclidean velocities $\dot{x}(t)$, $t\in [0,1]$ of any geodesic
$x:[0,1]\to \bar{O}$ are bounded by a constant $V(\delta)$ which
goes to zero for $\delta \to 0$. Since $\Gamma^\mu_{\alpha \beta}$
is bounded in a neighborhood of $r$, the geodesic equation implies
that there is $M>0$ such that $\frac{\dd v^\mu}{\dd t}\le M \Vert
v\Vert^2$, which is the same as saying that the (Euclidean) radius
of curvature is greater than $1/M$ and hence $\dot{x}$ can vary in
an subinterval of $[0,1]$ of an angle of at most $V(\delta) M$. Thus
if we take $\delta$ sufficiently small we can make the variation of
angle on the tangent to any geodesic $x:[0,1]\to \bar{O}$ to be
bounded by $\frac{1}{8}$ rad.

We have already shown in Theorem \ref{gux} that $D^2_q$ is $C^{1,1}$
on $C$ with a differential at $r$ given by $2g(P(q,r),\cdot)$. Since
the differential is continuous, $D^2_q$ is strongly differentiable
thus for every $\beta$ we can write  for sufficiently small
$\delta>0$,  for $q_1,q_2\in \bar{O}$
\[
\vert D^2_q(q_2)-D^2_q(q_1)-2g(P(q,r),q_2-q_1)\vert\le  \beta
{\sqrt{-g(P(q,r),P(q,r))}} \,\Vert q_2-q_1 \Vert.
\]
If $q_1$ and $q_2$ belong to the same level set, recalling that
\[
\dot \gamma(t_r)= P(q,r)/\sqrt{-g(P(q,r),P(q,r))}
\]
we obtain
\[
\vert g_r(\dot \gamma(t_r),q_2-q_1)\vert\le \beta
 \, \Vert q_2-q_1 \Vert.
\]
As $\beta\to 0$, the direction of $q_2-q_1$ is constrained to
approach $S(t_r)$. As the space of directions is a sphere and hence
compact there is a value of $\beta$ and a corresponding $O$ such
that whenever $q_1,q_2\in \bar{O}$ belong to the same level set of
$D_q$, $q_2-q_1$ makes with $S(t_r)$ an (Euclidean) angle smaller
than $1/16$ rad.

Since $\Vert P(q_1,q_2)-(q_2-q_1)\Vert \le \epsilon \Vert
q_2-q_1\Vert$ and $\epsilon <1/9$, the vector $P(q_1,q_2)$ makes
with $q_2-q_1$ an Euclidean angle smaller than $2 \arcsin
(\epsilon/2) <1/8$ rad (because $\epsilon<1/9$). Thus $P(q_1,q_2)$
makes with $S(t_r)$ an (Euclidean) angle smaller than $\frac{3}{16}$
rad, and with $L(t_r)$ an (Euclidean) angle smaller than
$\frac{2}{8}$ rad. If $x(a)$ and $x(b)$ are any two points in the
geodesic $x$ joining $q_1$ to $q_2$ then $P(x(a),x(b))$ makes an
angle with $P(q_1,q_2)$ of at most $\frac{1}{8}$ rad, thus it makes
an angle with  $L(t_r)$  smaller than $\frac{3}{8}$ rad. Moreover,
\[
\arctan([\frac{1-2\sqrt{\epsilon}}{1+2\sqrt{\epsilon}}]^{1/2})>
\arctan([\frac{1-2/3}{1+2/3}]^{1/2})>\frac{3}{8} \ \textrm{rad}
\]
thus $P(x(a),x(b))$ is $\eta^+$-spacelike (we knew already that it
was $g$-spacelike). Since $P(x(a),x(b))$ is $\eta^+$-spacelike we
have using Eq.\ (\ref{par}) and Eq. (\ref{jwj})
\begin{align*} \label{pat}
&\vert \nabla_{(b-a) \dot{x}(b)}D^2_q-\nabla_{(b-a) \dot{x}(a)}
D^2_q-2D^2(x(a),x(b)) \vert\\ &\quad \le
\frac{\sqrt{\epsilon}}{(1-\epsilon)^2}\,
g_{x(b)}(P(x(a),x(b)),P(x(a),x(b))) =
\frac{\sqrt{\epsilon}}{(1-\epsilon)^2}\,D^2(x(a),x(b)) ,
\end{align*}
which a redefinition of $\epsilon$, and a redefinition of
parametrization such that $(b-a)\to D(x(a),x(b))$ (namely the
$g$-arc length parametrization) brings to the form of Eq.\
(\ref{lor}).

The statement concerning the strong convexity of $D^2_q\circ x$
is immediate from the equivalences
recalled in Sect.\ \ref{nof}.

Let us prove the strict convexity of
$(D^2_q)^{-1}((-\infty,c))\cap O$. Suppose that $c<0$ is such that
$G:=(D^2_q)^{-1}((-\infty,c))\cap O\ne \emptyset$, otherwise there
is nothing to prove. Let $x:[a,b]\to C$ be a geodesic such that
$x(a),x(b)\in \bar{G}$, then $x(a),x(b)\in \bar{O}$, and since $O$
is strictly geodesically convex $x$ is contained in $O$ but for the
endpoints. Since $D^2_q:C\times C\to \mathbb{R}$ is continuous,
$x(a),x(b)\le c$ and we have to show that for every $t\in (a,b)$,
$D^2_q(x(t))<c$.

Suppose not, then there is some $t_{max}\in (a,b)$ such that
$D^2_q(x(t_{max}))$ is the maximum of $D^2_q(x(\cdot))$ over $[a,b]$
and $D^2_q(x(t_{max}))\ge c$. In particular,
\[
\frac{\dd (D^2_p\circ x)}{\dd t}\vert_{t=t_{max}}=\nabla_{\dot{x}(t_{max})} D^2_q(x(t_{max}))=0.
\]
However, by Eq.\
(\ref{lor}) no two values of $t$ can attain this maximum, thus in
any neighborhood $E$  of $t_{max}$  we can find $t_1,t_2\in
E\backslash \{t_{max}\}$, $t_1<t_{\max}<t_2$ such that
$D^2_q(x(t_1))=D^2_q(x(t_2))<D^2_q(x(t_{max}))$, thus we can apply
once again the equation following Eq.\ (\ref{lor}) using these two values as endpoints of
the spacelike geodesic. But that equation implies that
$D^2_q(x(t_{max}))< D^2_q(x(t_2))$,  a contradiction.
\end{proof}

\begin{lemma} \label{kji}
Let $p\in M$ and let $\gamma: I\to M$, $t\mapsto \gamma(t)$, be a
timelike geodesic such that $p=\gamma(0)$. The convex normal set
$C\ni p$ can be taken sufficiently small that once $I$ is redefined to be the
connected component of $\gamma^{-1}(C)$ containing $0$, the
following property holds. We can find $q_1=\gamma(t_1)$, and
$q_2=\gamma(t_2)$ with $t_1<0<t_2$, and a strictly convex normal set
$O\ni p$, such that introduced the constants $c_1:=D^2_{q_1}(p)$ and
$c_2:=D^2_{q_2}(p)$, we have that, for any $c_1'>c_1$ sufficiently
close to $c_1$ and for any $c_2'>c_2$ sufficiently close to $c_2$,
\[
S(c_1',c_2')=(D^2_{q_1})^{-1}((-\infty, c_1'))\cap
(D^2_{q_2})^{-1}((-\infty, c_2')) \cap O
\]
is strictly convex normal and globally hyperbolic.
\end{lemma}

\begin{proof}
By Theorem \ref{jse} we can  find a strictly convex relatively
compact set $C\ni p$ with the property of that theorem. In
particular, let $q_1=\gamma(t_1)$, $q_1\in C$, $t_1<0$, and let
$q_2=\gamma(t_2)$, $q_2\in C$, $t_2>0$. There is a strictly convex
normal set $O_1\ni p$, $\bar{O}_1\subset I^+_C(q_1)$, such that
$D^2_{q_1}:C\times C\to \mathbb{R}$ is strongly convex over the
geodesics segments in $O_1$ connecting two points in its level
surfaces. Similarly there is a strictly convex set $O_2$,
$\bar{O}_2\subset I^{-}_C(q_2)$, with an analogous property with
respect to $q_2$. Let $O=O_1\cap O_2$.

Let us introduce the closed sets \[A(c_1',c_2')=
(D^2_{q_1})^{-1}((-\infty, c_1'])\cap (D^2_{q_2})^{-1}((-\infty,
c_2'])\cap \bar{O}.\] The set $A(c_1,c_2)$ contains only the point $p$ for
otherwise there would be a different timelike curve composed of two
geodesic pieces of total Lorentzian length equal with  that of
$\gamma \vert_{[t_{1},t_2]}$, a contradiction to Theorem
\ref{pwa}. Since the intersection of a family of non-empty compact
sets is non-empty, for sufficiently large $i$ the set
$A(c_1+1/i,c_2+1/i)$ must be disjoint from $\p O$.

Thus for $c_1'\le c_1+1/i$ and $c_2'\le c_2+1/i$, $S(c_1',c_2')$ is
the component of the open set $(D^2_{q_1})^{-1}((-\infty, c_1'))\cap
(D^2_{q_2})^{-1}((-\infty, c_2'))$ contained in $O$ and its closure
is contained in $O$. As $D^2_{q_1}$ is decreasing over future
directed causal curves and $D^2_{q_1}$ is increasing over future
directed causal curves, no causal curve in $C$ can leave and reenter
$S(c_1',c_2')$, in particular since $C$ is causally simple and
relatively compact, $S(c_1',c_2')$ is globally hyperbolic.

The set $S(c_1',c_2')$ is strictly convex normal because it is the
intersection of the strictly convex normal set
$(D^2_{q_1})^{-1}((-\infty, c_1'))\cap O_1$ and the strictly convex
normal set $(D^2_{q_2})^{-1}((-\infty, c_2'))\cap O_2$.
\end{proof}

\begin{proof}[Proof of Corollary \ref{cop}]
Just take $C_i=S(c_1+1/(k+i),c_2+1/(k+i))$ for sufficiently large
$k$ and use the results of Lemma \ref{kji} (see also its proof).
\end{proof}

\begin{acknowledgements}
 After this paper was posted on the preprint archive  (1308.6675) another work by M. Kunzinger, R. Steinbauer, M. Stojkovi\'c and J. A. Vickers was posted which obtained some of the results contained in this work through a regularization scheme for $C^{1,1}$ metrics \cite{kunzinger13b}. The approaches of this and that  work nicely complement each other. However, the  results of this work seem more general (e.g.\ they apply to Finsler geometry) and stronger (e.g.\ we prove strong differentiability of the exponential map at he origin). Also the proofs we provide are technically more elementary. The reader is referred to \cite{kunzinger13b} for a comparison. I thank Piotr
Chru{\'s}ciel  for some useful suggestions.


\end{acknowledgements}



\def\cprime{$'$}

\end{document}